\documentclass[12pt]{article}
\usepackage{amsmath}
\usepackage{amsfonts}
\usepackage{amssymb}
\usepackage{amsthm}
\usepackage{graphicx}
\usepackage{fullpage}
\usepackage{setspace}
\usepackage{verbatim}
\usepackage[numbers]{natbib}
\usepackage{appendix}
\usepackage{rotating}
\usepackage{multirow}

\newtheorem{theorem}{Theorem}[section]

\newtheorem{lemma}{Lemma}[section]

\theoremstyle{definition}

\newcommand{\FHol}{\mathcal{F}_{\textnormal{H\"{o}l}}} %
\newcommand{\Fselfsim}{{\mathcal{F}_{\textnormal{self-sim}}}} %
\newcommand{\Fselfsimalt}{{\overline{\mathcal{F}}_{\textnormal{self-sim}}}} %
\newcommand{\FGN}{{\mathcal{F}_{\textnormal{GN}}}} %
\newcommand{\length}{{\textnormal{length}}}
\newcommand{\dtest}{{d_{\textnormal{test}}}}
\newcommand{\Ksupp}{{C_K}}
\newcommand{\cvalconst}{{\bar c_K}}
\newcommand{\psisupp}{{C_\psi}}

\title{Adaptation Bounds for Confidence Bands under Self-Similarity}
\author{Timothy B. Armstrong\thanks{email: timothy.armstrong@yale.edu.  Thanks
    to Richard Nickl helpful comments and discussion.}  \\
Yale University}
\date{\today}

\begin{document}

\maketitle

\begin{abstract}
We derive bounds on the scope for a confidence band to adapt to the unknown
regularity of a nonparametric function that is observed with noise, such as a regression function or density, under the
self-similarity condition proposed by Gin\'{e} and Nickl \cite{gine_confidence_2010}.  
We find that adaptation can only be achieved up to a term that depends
on the choice of the constant used to define self-similarity, and that this term becomes
arbitrarily large for conservative choices of the self-similarity constant.
We construct a confidence band that achieves this bound, up to a constant term
that does not depend on the self-similarity constant.
Our results suggest that care must be taken in choosing and interpreting the
constant that defines self-similarity, since the dependence of adaptive confidence
bands on this constant cannot be made to disappear asymptotically.
\end{abstract}

\section{Introduction}

Consider the problem of constructing a confidence band for a function that is
observed with noise, such as a regression function or density.  It will be convenient to state our
results in the white noise model
\begin{align*}%
  Y(t)=\int_0^t f(s)\, ds+\sigma_n W(t),
  \quad \sigma_n=\sigma/\sqrt{n} %
\end{align*}
which maps to the regression or density setting 
with $n$ playing the
role of sample size \cite{brown_asymptotic_1996,nussbaum_asymptotic_1996}.
Here $f:\mathbb{R}\to\mathbb{R}$ is an unknown function, $W(t)$ is a standard
Brownian motion and $Y(t)$ is observed with $\sigma_n$ treated as known.  
To obtain good estimates and confidence
bands, one must impose some regularity on the function $f$.  This is typically
done by assuming that $f$ is in a derivative smoothness class, such as the
H\"{o}lder class $\FHol(\gamma,B)$, which
formalizes the notion that the $\gamma$th derivative is bounded by $B$:
\begin{align*}%
  \FHol(\gamma,B)=\{f \colon \textnormal{ for all }t,t'\in \mathbb{R},  |f^{(\lfloor \gamma\rfloor)}(t)-f^{(\lfloor \gamma\rfloor)}(t')|\le B |t-t'|^{\gamma-\lfloor \gamma\rfloor}\}
\end{align*}
where $\lfloor \gamma \rfloor$ denotes the greatest integer strictly less than $\gamma$.
We are interested in constructing a confidence band for $f$ on an interval,
which we take to be $[0,1]$.
A confidence band is a collection of random intervals 
$\mathcal{C}_n(x)=\mathcal{C}_n(x;Y)$ for $x\in[0,1]$ that depend on the data $Y$ observed at noise level $\sigma_n=\sigma/\sqrt{n}$.
Following the standard definition, we say that $\mathcal{C}_n(\cdot)$ is a
confidence band with coverage $1-\alpha$ over the class $\mathcal{F}$ if
\begin{align}\label{coverage_eq}
  \inf_{f\in\mathcal{F}}P_f\left(\text{for all }x\in [0,1],\, f(x)\in \mathcal{C}_n(x)\right)
  \ge 1-\alpha
\end{align}
where $P_f$ denotes probability when $Y(t)$ is drawn according to $f$.  Although we focus on the interval $[0,1]$, to avoid boundary issues, we will
assume that $Y(t)$ is observed on
an interval $[-\eta,1+\eta]$ for some $\eta>0$.

Using knowledge of the class $\FHol(\gamma,B)$, one can construct estimators and confidence bands that are near-optimal in a minimax sense.  In practice, however, it can
be difficult to specify $\gamma$ and $B$ a priori.  This has led to the
paradigm of adaptation: one seeks estimators and confidence bands that are
nearly optimal for all $\gamma$ and $B$ in some range
without a priori knowledge of $\gamma$ or $B$.  Such procedures 
are called ``adaptive.''
Unfortunately, while it is possible to construct estimators that adapt to the
unknown value of $\gamma$ and $B$,
(see \cite{tsybakov_pointwise_1998} and references therein), it follows from \cite{low_nonparametric_1997} that adaptive confidence band 
construction over derivative smoothness classes is impossible.

To recover the possibility of adaptive confidence band construction,
\cite{gine_confidence_2010} propose
an additional condition known as ``self-similarity'' (see also \cite{picard_adaptive_2000}), which uses a
constant $\varepsilon>0$ to rule out functions such that the level of regularity
is statistically difficult to detect.  Imposing these additional conditions
leads to a class $\Fselfsim(\gamma,B,\varepsilon)\subsetneq
\FHol(\gamma,B)$.
\cite{gine_confidence_2010} derive confidence bands that are rate-adaptive to
the unknown parameter $\gamma$ over these smaller classes, and they show that the set
$\FHol(\gamma,B)\backslash
\cup_{\varepsilon>0}\Fselfsim(\gamma,B,\varepsilon)$ of functions ruled out by
this assumption (as $\varepsilon\to 0$) is small in a certain topological sense.
A subsequent literature has further examined the use of self-similarity and
related assumptions in forming adaptive confidence bands (see references below).

These results provide a promising approach to constructing a
confidence band such that the width reflects the unknown
regularity $\gamma$ of the function $f$.
However, these confidence bands require a priori knowledge of other 
regularity parameters, including
$\varepsilon$, either explicitly or through unspecified constants and
sequences that must be chosen in a way that depends on $\varepsilon$ 
in order to guarantee coverage for a given sample size or noise level.
Furthermore, these choices have a first order asymptotic effect on the width of
the confidence band, and making an asymptotically conservative choice by taking
$\varepsilon=\varepsilon_n\to 0$ leads to a slightly slower rate of convergence.
This has led to concern about whether self-similarity assumptions can lead
to a ``practical'' approach to confidence band construction (see, for example, the discussion
on pp. 2388-2389 of \cite{hoffmann_adaptive_2011}): while
self-similarity removes the need to specify the order $\gamma$ of the
derivative, currently available methods still require specifying other
regularity parameters.
Can one construct a confidence band that is fully adaptive without specifying
any of the regularity parameters $\gamma$, $B$ or $\varepsilon$?

An implication of the results in this paper is that it is impossible to achieve
such a goal.
In particular, we show that a confidence band that is adaptive over
classes $\Fselfsim(\gamma,B,\varepsilon)$ over a range of $\gamma$ or $B$ must
necessarily pay an adaptation penalty proportional to
$\varepsilon^{-1/(2\gamma+1)}$.  As a consequence, adaptive confidence bands in
self-similarity classes require explicit specification of the self-similarity
constant $\varepsilon$, and taking
$\varepsilon=\varepsilon_n\to 0$ requires paying a penalty in the rate.
On a more positive note, once $\varepsilon$ is given, we construct a confidence
band that is ``practical'' in the sense that it is valid for a fixed
sample size or noise level in Gaussian settings, and it does not depend
on additional unspecified constants or sequences once $\varepsilon$ is given.

To describe these results formally,
let $\mathcal{I}_{n,\alpha,\mathcal{F}}$ denote the set of confidence bands
that satisfy the coverage requirement (\ref{coverage_eq}).
Subject to this coverage requirement, we compare worst-case length of $\mathcal{C}_n$ over a possibly smaller class $\mathcal{G}$.  Letting $\length(\mathcal{A})=\sup \mathcal{A}-\inf\mathcal{A}$ denote the length of a set $\mathcal{A}$, let
\begin{align*}
  R_\beta(\mathcal{C}_n;\mathcal{G})
  =\sup_{f\in \mathcal{G}} q_{\beta,f}\left(\sup_{x\in[0,1]}\length(\mathcal{C}_n(x))\right)
\end{align*}
where $q_{\beta,f}$ denotes the $\beta$ quantile when $Y\sim f$.
Following \cite{cai_adaptation_2004}, define
\begin{align*}
  R^*_{n,\alpha,\beta}(\mathcal{G},\mathcal{F})
  =\inf_{\mathcal{C}_n(\cdot)\in \mathcal{I}_{n,\alpha,\mathcal{F}}} R_\beta(\mathcal{C}_n;\mathcal{G})
\end{align*}
to be the optimal worst-case length over $\mathcal{G}$ of a band with coverage
over $\mathcal{F}$, where $\mathcal{G}\subseteq\mathcal{F}$.
A minimax confidence band over the set $\mathcal{F}$ is one that achieves the
bound $R^*_{n,\alpha,\beta}(\mathcal{F},\mathcal{F})$.
Given a family $\mathcal{F}(\tau)$ of function classes indexed by a regularity
parameter $\tau\in\mathcal{T}$, the goal of adaptive confidence band construction
is to find a single confidence band $\mathcal{C}_n(\cdot)$ that is close to
achieving this bound for each $\mathcal{F}(\tau)$, while also maintaining coverage
$1-\alpha$ for each
$\mathcal{F}(\tau)$ (so that
$\mathcal{C}_n(\cdot)\in\mathcal{I}_{n,\alpha,\cup_{\tau\in\mathcal{T}}\mathcal{F}(\tau)}$).
Suppose that a confidence band $\mathcal{C}_n(\cdot)\in\mathcal{I}_{n,\alpha,\cup_{\tau\in\mathcal{T}}\mathcal{F}(\tau)}$ achieves this goal up to a
factor $A_n(\tau)$:
\begin{align*}
  R_\beta(\mathcal{C}_n;\mathcal{F}(\tau)) 
  \le A_n(\tau)R^*_{n,\alpha,\beta}(\mathcal{F}(\tau),\mathcal{F}(\tau))
  \text{ all } \tau\in\mathcal{T}
\end{align*}
(in the present setting, $A_n(\tau)$ will not depend on $\alpha$ or $\beta$ once
$n$ is large enough).
We will call such a band \emph{adaptive} to $\tau$ up to the \emph{adaptation
  penalty} $A_n(\tau)$.
If the adaptation penalty is bounded as a function of $n$, we will say that the
confidence band is (rate) adaptive (this corresponds to what
\cite{cai_adaptation_2004} call ``strongly adaptive'').
Note that
$R^*_{n,\alpha,\beta}(\mathcal{F}(\tau),\cup_{\tau\in\mathcal{T}}\mathcal{F}(\tau))\allowbreak
/R^*_{n,\alpha,\beta}(\mathcal{F}(\tau),\mathcal{F}(\tau))$
provides a lower bound for the adaptation penalty of any confidence band $\mathcal{C}_n(\cdot)$.

\begin{sloppypar}
For H\"{o}lder classes,
$R^*_{n,\alpha,\beta}(\FHol(\gamma,B),\FHol(\gamma,B))$ decreases at the
rate $(n/\log n)^{-\gamma/(2\gamma+1)}$.
A confidence band that is rate adaptive to $\gamma$
would achieve this rate simultaneously for all $\gamma$ in some set
$[\underline\gamma,\overline\gamma]$ while maintaining coverage over $\cup_{\gamma\in[\underline\gamma,\overline\gamma]}\FHol(\gamma,B)$.
However, as noted above, the results of \cite{low_nonparametric_1997} imply
that this is impossible.
Indeed,
$R^*_{n,\alpha,\beta}(\FHol(\gamma,B),\cup_{\gamma'\in[\underline\gamma,\overline\gamma]}\FHol(\gamma',B))$
decreases at the rate $(n/\log n)^{-\underline\gamma/(2\underline\gamma+1)}$ for
each $\gamma\in[\underline\gamma,\overline\gamma]$, so the adaptation penalty
for H\"{o}lder classes is of order
$(n/\log n)^{\gamma/(2\gamma+1)-\underline\gamma/(2\underline\gamma+1)}$, which is quite severe. 
\end{sloppypar}

To salvage the possibility of adaptation, \cite{gine_confidence_2010} propose
augmenting the H\"{o}lder condition with an auxiliary condition.
Let $K:\mathbb{R}^2\to \mathbb{R}$ be a function, called a kernel, such that
$x\mapsto K(t,x)$ is of bounded variation for each $t$.  Let
$K_j(t,x)=2^jK(2^jt,2^jx)$ for any integer $j$, and let
$\hat f(t,j)=\int K_j(t,x)\, dY(x)$.
This allows for convolution kernels $K(t,x)=\tilde K(t-x)$ (in which case
$2^{-j}$ is the bandwidth) and wavelet projection kernels $K(t,x)=\sum_k
\phi(t-k)\phi(x-k)$
(in which case $\phi$ is the father wavelet and $j$ is the resolution level).
Let $K_jf(t)=\int K_j(t,x)f(x)\, dx$.
Note that $E_f\hat f(t,j)=K_jf(t)$, where $E_f$ denotes expectation when $Y(x)$
is drawn according to $f$, so that the bias is given by $K_jf(t)-f(t)$.
Under appropriate conditions on $K$, an upper bound on this bias for functions
in $\FHol(\gamma,B)$ follows from standard calculations
(see \cite[][Ch. 4]{gine_mathematical_2015}):
\begin{align}\label{tildeCgamma_bound_eq}
  \sup_{t\in[0,1]} |K_jf(t)-f(t)|\le \tilde C B 2^{-j\gamma}
\end{align}
for some constant $\tilde C$.
\cite{gine_confidence_2010} impose such a bound on bias directly, along
with an analogous lower bound.  For $\underline j,b_1,b_2>0$, let
$\FGN(\gamma,b_1,b_2)=\FGN(\gamma,b_1,b_2;K,\underline j)$ denote the set of
functions $f$ satisfying Condition 3 of \cite{gine_confidence_2010}: for all integers
$j\ge\underline j$,
\begin{align}\label{GN_j_condition}
b_12^{-j\gamma}\le \sup_{t\in [0,1]}\left|K_jf(t)-f(t)\right|
\le b_2 2^{-j\gamma}.
\end{align}
Since we will also be imposing H\"{o}lder conditions, which, as noted above,
satisfy the upper bound with $b_2=\tilde C B$, it is natural to
make the lower bound proportional to $B$ as well, by taking $b_1=\varepsilon B$
for some $\varepsilon>0$.  To this end, let
$\Fselfsim(\gamma,B,\varepsilon)=\Fselfsim(\gamma,B,\varepsilon;K,\underline j)$
be the set of functions in $\FHol(\gamma,B)$ such that the lower bound in
(\ref{GN_j_condition}) holds with $b_1=\varepsilon B$ for all integers $j\ge
\underline j$.
By the discussion above, this is equivalent to defining
$\Fselfsim(\gamma,B,\varepsilon;K,\underline j)=\FHol(\gamma,B)\cap
\FGN(\gamma,\varepsilon B,C B;K,\underline j)$ for any $C\ge \tilde C$.
We will refer to $\varepsilon$ as a ``self-similarity constant,'' and we will
call the class $\Fselfsim$ a ``self-similarity class.''
Note that, by defining $\varepsilon$ to be (up to a constant) the ratio of the
upper and lower bounds on the bias, we are separating the role of
self-similarity and the smoothness constant.
In particular, the self-similarity constant is scale invariant.
See Section \ref{alternative_self_sim_sec} for alternative formulations of the
notion of a ``self-similarity constant.''

\begin{sloppypar}
Our main results are efficiency bounds that have implications for the adaptation
penalty $A_n(\gamma,B)$ for confidence bands that adapt to the
regularity parameters $(\gamma,B)$ over a rich enough set $\mathcal{T}$ in the
self-similarity class $\Fselfsim(\varepsilon,\gamma,B)$.  In particular, our
results imply the existence of a constant $C_{*}>0$ such that, for
large enough $n$, the adaptation penalty for any confidence band must satisfy
the lower bound
$C_*\varepsilon^{-1/(2\gamma+1)}<A_n(\gamma,B)$.  Furthermore, we
construct a confidence band with adaptation penalty
$A_n(\gamma,B)<C^*\varepsilon^{-1/(2\gamma+1)}$, where $C^*<\infty$
(the constants $C_*$ and $C^*$ do not depend on $\varepsilon$ but may depend on the set $\mathcal{T}$ over which
adaptation is required).
For the lower bounds, we consider separately the cases of adaptation to $B$ with
$\gamma$ known (i.e. $\mathcal{T}=\gamma\times [\underline B,\overline B]$) and
adaptation to $\gamma$ with $B$ known (i.e.
$\mathcal{T}=[\underline\gamma,\overline\gamma]\times B$).  In both cases, the
lower bound gives the same $\varepsilon^{-1/(2\gamma+1)}$ term.
We also consider the possibility of ``adapting to the self-similarity constant''
and find that that this is not possible: if we allow $\varepsilon$ to be in some
set $[\underline\varepsilon,\overline\varepsilon]$, then we obtain a lower bound
proportional to $\underline\varepsilon^{-1/(2\gamma+1)}$.
\end{sloppypar}

Our results relate to the literature deriving confidence bands under
self-similarity conditions.
\cite{gine_confidence_2010} propose a confidence band
that has coverage over $f\in\Fselfsim(\gamma,B,\varepsilon_n)$ for a range of $(\gamma,B)$, where
$\varepsilon_n\to 0$ with the sample size, and they show that it is adaptive up
to a penalty $A_n(\gamma,B)$ where $A_n(\gamma,B)\to\infty$ slowly with the sample size
$n$.  Our lower bounds show that a penalty of this form is unavoidable if one
takes $\varepsilon_n\to 0$.  \cite{bull_honest_2012} and \cite{chernozhukov_anti-concentration_2014}
propose confidence bands with coverage over self-similarity classes with
$\varepsilon$ fixed, and they show that these confidence bands are fully rate
adaptive (i.e. the adaptation penalty $A_n(\gamma,B)$ is bounded as $n$ increases).
Checking whether the adaptation penalty for these confidence bands takes the
optimal form $C^*\varepsilon^{-1/(2\gamma+1)}$ for small $\varepsilon$ appears
to be difficult, and we derive upper bounds using a different confidence band
(although the confidence band we propose builds on ideas in these papers; see
Section \ref{discussion_sec}).

To our knowledge, this paper is the first to derive lower bounds
on adaptation constants for
confidence bands under
self-similarity conditions.
A related question, addressed by \cite{hoffmann_adaptive_2011}
and \cite{bull_honest_2012}, is whether the self-similarity conditions
themselves can be weakened.
These papers derive lower bounds showing that certain ways of relaxing
self-similarity necessarily lead to a penalty in the rate, and our finding that
taking $\varepsilon=\varepsilon_n\to 0$
requires paying such a penalty
complements these results.
In addition, a large literature has considered adaptive confidence sets in related settings under conditions that are similar
to the self-similarity condition used by \cite{gine_confidence_2010}.
In the Gaussian
sequence setting, \cite{szabo_frequentist_2015} propose a condition called a
``polished tail'' condition.
They use this condition to show frequentist coverage of adaptive Bayesian credible sets
(see also \cite{sniekers_adaptive_2015,van_der_pas_uncertainty_2017}).
Other applications of self-similarity type conditions include
high dimensional sparse regression \cite{nickl_confidence_2013},
density estimation on the sphere \cite{kueh_locally_2012},
locally adaptive confidence bands \cite{patschkowski_locally_2019},
binary regression \cite{mukherjee_optimal_2018}
and
$L_p$ confidence sets
\cite{bull_adaptive_2013,carpentier_honest_2013,nickl_sharp_2016} (in contrast
to our setting where $p=\infty$, some range of adaptation is possible even
without self-similarity when $p<\infty$; see \cite{juditsky_nonparametric_2003,robins_adaptive_2006,cai_adaptive_2006}).
Self-similarity is also related to ``signal strength'' conditions used in other
settings, such as ``beta-min'' conditions used to study variable selection in
high dimensional regression (see \cite{buhlmann_statistics_2011}, Section 7.4).

Our lower bounds apply immediately
to confidence bands with coverage under any set $\mathcal{F}$ that weakens
the self-similarity conditions in \cite{gine_confidence_2010}.
This includes,
for certain ranges of regularity constants,
the conditions used in
\cite{bull_honest_2012} and, for adaptation to $B$ with $\gamma$ fixed,
\cite{hoffmann_adaptive_2011}.
\cite{szabo_frequentist_2015} show that their conditions are weaker than a
natural definition of self-similarity in the Gaussian sequence setting.
A full characterization of upper and lower bounds in these and other related
settings is left for future research.

\section{Adaptation Bounds for Self-Similar Functions}\label{adaptation_bounds_sec}

This section states our main results.
We first give lower bounds for
adaptation, separating the role of adaptation to the constant $B$ and the
exponent $\gamma$.  We then construct a confidence band that achieves these
bounds, up to a constant that does not depend on the self-similarity constant $\varepsilon$, simultaneously for all $\gamma$ and $B$ on bounded intervals.
Finally, we provide lower bounds for an alternative formulation of the problem,
and a discussion of our results.

Before stating the formal results, we give a heuristic explanation of the
bounds.
Self-similarity allows for adaptation by bounding the bias at a scale $j_1$
using an estimate of the bias at a different scale $j_2$:
the bias
$\sup_{t\in
  [0,1]}\left|K_{j_1}f(t)-f(t)\right|$ of $\hat f(t,j_1)$
is bounded by
$\varepsilon^{-1} \tilde C 2^{-\gamma(j_1-j_2)} \sup_{t\in
  [0,1]}\left|K_{j_2}f(t)-f(t)\right|$.
If we can get an estimate of this upper bound that
converges more quickly than the estimation error in $\hat f(t,j_1)$ (which turns
out to be possible by taking $j_2$ to increase slightly more slowly than $j_1$),
then we can treat this upper bound as known.  Since $\sup_{t\in
  [0,1]}\left|K_{j_2}f(t)-f(t)\right|$ is
bounded by $\tilde C B 2^{-\gamma j_2}$, this is as good as using the bound
$\varepsilon^{-1} \tilde C^2 B 2^{-\gamma j_1}$
on the bias of $\hat f(t,j_1)$.
Choosing $j_1$ to balance this term with the estimation error in
$\sup_{t\in[0,1]}|\hat f(t,j_1)-K_{j_1}(f)|$ then gives the rate with the
$\varepsilon^{-1/(2\gamma+1)}$ factor.
Note that this gives the same rate and constant as using prior knowledge of the
H\"{o}lder class, but replacing $B$ with $\varepsilon^{-1}B$, up to a constant
that does not depend on $\varepsilon$, $\gamma$ or $B$.

The constructive upper bound in Section \ref{achieving_bounds_sec_main} below uses a confidence band that
formalizes these ideas.  The lower bounds in Section \ref{lower_bounds_sec} show formally that
no further information can be used to improve this confidence band, up to
factors that do not depend on $\varepsilon$, $\gamma$ or $B$.

\subsection{Lower Bounds}\label{lower_bounds_sec}

We now give bounds for adaptation over the classes
$\Fselfsim(\gamma,B,\varepsilon)$.  Proofs of the lower bounds in this section
are given in Section \ref{lower_bounds_proofs_sec}.
We impose the following conditions on the kernel $K$:
\begin{align}\label{kernel_lower_bound_assump}
  \parbox{4.5in}{there exists $\Ksupp<\infty$ such that $K(y,x)=0$ for $|x-y|>\Ksupp$ and, for all $k\in\mathbb{Z}$ and $x,y\in\mathbb{R}$, $K(y,x)=K(y-k,x-k)$.}
\end{align}
These conditions hold for convolution kernels with finite
support, and for wavelet projection kernels for which the father wavelet has
bounded support.

We first consider adaptation to the constant $B$.

\begin{theorem}\label{constant_adaptation_thm_general_kernel}
  Let $\gamma>0$ and let
  $0<2\alpha<\beta<1$.  Let $K$ be a kernel satisfying
  (\ref{kernel_lower_bound_assump}).  There exists $\underline j_{K,\gamma}$,
  $C_{K,\gamma,*}>0$ and $\eta_{K,\gamma}>0$ such that, for any
  $0<\underline B\le B\le \overline B$,
  $\varepsilon\le
  \varepsilon'<\eta_{K,\gamma}$ and $\underline\ell\ge \underline j_{K,\gamma}$,
  \begin{align*}
    &R^*_{n,\alpha,\beta}(\Fselfsim(\gamma,B,\varepsilon';K,\underline\ell), 
\cup_{B'\in[\underline B,\overline
  B]}\Fselfsim(\gamma,B',\varepsilon;K,\underline \ell))  \\
    &\ge (1+o(1))C_{K,\gamma,*}\min\{\varepsilon^{-1}B,\overline B\}^{1/(2\gamma+1)}\left(\sigma_n^2\log (1/\sigma_n)\right)^{\gamma/(2\gamma+1)}.
  \end{align*}
\end{theorem}

We now consider adaptation to $\gamma$ with $B$ known.  To avoid notational clutter, we normalize $B$ to one.

\begin{theorem}\label{exponent_adaptation_thm_general_kernel}
  Let $0<\underline\gamma<\gamma\le \overline\gamma$ and let
  $0<2\alpha<\beta<1$.  Let $K$ be a kernel that satisfies
  (\ref{kernel_lower_bound_assump}).  There exist $C_{K,\overline\gamma,*}$,
  $\underline j_{K,\overline\gamma}$ and $\eta_{K,\overline\gamma}$ depending
  only on $K$ and $\overline\gamma$ such that, for all $\underline\ell\ge
  \underline j_{K,\overline\gamma}$ and $0<\varepsilon\le\varepsilon'<\eta_{K,\overline\gamma}$,
  \begin{align*}
    &R^*_{n,\alpha,\beta}\left( \Fselfsim(\gamma,1,\varepsilon'; K,\underline\ell),
      \cup_{\gamma'\in[\underline\gamma,\overline\gamma]} \Fselfsim(\gamma',1,\varepsilon; K,\underline\ell)\right)  \\
    &\ge (1+o(1))C_{K,\overline\gamma,*} \varepsilon^{-1/(2\gamma+1)}\left(\sigma_n^2\log (1/\sigma_n)\right)^{\gamma/(2\gamma+1)}.
  \end{align*}
\end{theorem}

It follows from Theorems \ref{constant_adaptation_thm_general_kernel} and
\ref{exponent_adaptation_thm_general_kernel} that adaptive confidence bands must pay an
adaptation penalty proportional to $\varepsilon^{-1/(2\gamma+1)}$.  Furthermore,
these results show that one cannot ``adapt to the self-similarity constant:''
if we require coverage for $\varepsilon$-self-similarity, then the adaptation
penalty is proportional to $\varepsilon^{-1/(2\gamma+1)}$, even for functions
that are $\varepsilon'$-self-similar with $\varepsilon'>\varepsilon$.

\subsection{Achieving the Bound}\label{achieving_bounds_sec_main}

We now turn to upper bounds.  Both of these bounds can be achieved simultaneously for all
$\gamma\in[\underline\gamma,\overline\gamma]$ and $B\in[\underline B,\overline
B]$ by a single confidence band, up to an additional term that depends only on
$K$ and the range $[\underline\gamma,\overline\gamma]$.
We first state the upper bound, and then describe the confidence band that
achieves it.

We make some additional assumptions on the kernel:
\begin{align}\label{kernel_upper_bound_assump}
  \parbox{3.7in}{$\sup_{t\in[0,1]}\int K(t,x)^2\, dx<\infty$ and there exists $\tau_K>0$ such that $\sup_{s,t\in[0,1]}\frac{\int \left[ K(s,x)-K(t,x) \right]^2\, dx}{|s-t|^{\tau_K}}<\infty$.}
\end{align}
Condition (\ref{kernel_upper_bound_assump}) is a mild continuity condition.  For
convolution kernels $K(y,x)=\tilde K(y-x)$ or wavelet projection kernels
$K(y,x)=\sum_k\phi(y-k)\phi(x-k)$, it is sufficient for the kernel $\tilde K$ or
father wavelet $\phi$ to be bounded with finite support and bounded first
derivative (see \cite{gine_confidence_2010}, p. 1146 for the latter case).

\begin{theorem}\label{upper_bound_thm_general_kernel}
Let $0<\underline B<\overline B$ and $0<\underline\gamma<\overline\gamma$ be
given, and let $K$ be a kernel that satisfies 
(\ref{kernel_lower_bound_assump}) and (\ref{kernel_upper_bound_assump}), such
that, for some $\tilde C$, (\ref{tildeCgamma_bound_eq}) holds for
all $B\in[\underline B,\overline B]$ and all $\gamma\in [\underline\gamma,\overline\gamma]$.  There exists a confidence band
$\mathcal{C}_n(\cdot)$ and a constant
$C_{K,\overline\gamma,\tilde C}^{*}$ depending only on $K$, $\overline\gamma$ and $\tilde
C$ such that, with probability approaching one uniformly
over $\cup_{\gamma\in[\underline\gamma,\overline\gamma]}\cup_{B\in [\underline
  B,\overline B]}\Fselfsim(\gamma,B,\varepsilon)$, 
\begin{align*}
  \sup_{x\in[0,1]}\length\left( \mathcal{C}_n(x) \right)
  \le C_{K,\overline\gamma,\tilde C}^{*}\left( B\varepsilon^{-1} \right)^{1/(2\gamma+1)}(\sigma^2_n\log (1/\sigma^2_n))^{\gamma/(2\gamma+1)}
\end{align*}
and $f(x)\in \mathcal{C}_n(x)$ all $x\in[0,1]$.
\end{theorem}

\begin{sloppypar}
To prove this theorem, we construct a confidence band that has
coverage for the class
$\cup_{B\in[\underline B,\overline B]} \cup_{\gamma\in[\underline \gamma,\overline \gamma]}
  \FGN(\gamma,\varepsilon B, B)$,
such that the width is bounded by a constant times
$(\varepsilon^{-1}B)^{1/(2\gamma+1)}(\sigma_n\log
(1/\sigma_n))^{\gamma/(2\gamma+1)}$ with probability approaching one uniformly
over the class $\FGN(\gamma,\varepsilon B, B)$.
Letting $\tilde\varepsilon=\varepsilon/\tilde C$ and $\tilde B=\tilde C B$, we have
$\Fselfsim(\varepsilon,\gamma,B)\subseteq\FGN(\gamma,\tilde\varepsilon \tilde B,
 \tilde B)$
under (\ref{tildeCgamma_bound_eq}),
so that the conclusion of Theorem
\ref{upper_bound_thm_general_kernel} holds for this confidence band, constructed
with $\tilde\varepsilon=\varepsilon/\tilde C$ in place of $\varepsilon$.
We describe the confidence band here, with additional details in Appendix
\ref{achieving_bounds_sec_append}.
\end{sloppypar}

Let $\Delta(j,j';f)=\sup_{x\in[0,1]}|K_jf(x)-K_{j'}f(x)|$
and
$\hat \Delta(j,j')=\sup_{x\in[0,1]}|\hat f(x,j)-\hat f(x,j')|$.
Let $c(j)$ and $\tilde c(j,j')$ be critical values satisfying
\begin{align}\label{fxh_coverage_eq}
|\hat f(x,j)-K_jf(x)|\le c(j)
\text{ all } x\in [0,1],\, j\in \mathcal{J}_n
\end{align}
and
\begin{align}\label{delta_coverage_eq}
  |\hat \Delta(j,j')-\Delta(j,j';f)|\le \tilde c(j,j')
  \text{ all } j,j'\in \mathcal{J}_n
\end{align}
with some prespecified probability
for all $f\in
\cup_{\gamma\in[\underline\gamma,\overline\gamma]}\cup_{B\in[\underline
B,\overline B]}\FGN(\gamma,\varepsilon B,B)$, where
$\mathcal{J}_n=\{\underline\ell_n,\underline\ell_n+1,\ldots,\overline \ell_n\}$ for
some $\underline\ell_n$, $\overline \ell_n$
(it suffices to set $c(j)=\cvalconst \sigma_n
2^{j/2}\sqrt{j}$ and $\tilde c(j,j')=c(j)+c(j')$ for a large enough constant
$\cvalconst$ and to take $\underline \ell_n\to\infty$ with $\underline
\ell_n/\log n\to 0$ and $\overline \ell_n/\log n\to\infty$; see Appendix
\ref{achieving_bounds_sec_append}).
We construct a confidence band that covers $f$ for all $f\in
\cup_{\gamma\in[\underline\gamma,\overline\gamma]}\cup_{B\in[\underline
B,\overline B]}\FGN(\gamma,\varepsilon B,B;K,\underline \ell_n)$ on the event
that (\ref{fxh_coverage_eq}) and (\ref{delta_coverage_eq}) both hold.
To this end, we use $\Delta(j,j';f)$ along with the self-similarity condition to
bound the bias $|K_jf(x)-f(x)|$.  This, along with the confidence bands $\hat
f(x,j)\pm c(j)$ and $\hat \Delta(j,j')\pm \tilde c(j,j')$ for $K_jf(x)$ and
$\Delta(j,j';f)$ leads to a confidence
band for $f$.
First, note that, for $f\in\FGN(\gamma,\varepsilon B,B;K,\underline \ell)$ and $j_1,j_2\ge \underline\ell$,
\begin{align}\label{delta_bound_eq}
  &B(\varepsilon 2^{-j_1\gamma}-2^{-j_2\gamma})
    \le \sup_{x\in[0,1]}|K_{j_1}f(x)-f(x)| - \sup_{x\in[0,1]}|K_{j_2}f(x)-f(x)|  \nonumber  \\
    &\le \Delta(j_1,j_2;f)  %
    \le \sup_{x\in[0,1]}|K_{j_1}f(x)-f(x)| + \sup_{x\in[0,1]}|K_{j_2}f(x)-f(x)|  \\
    &\le B(2^{-j_1\gamma}+2^{-j_2\gamma})  \nonumber  
\end{align}
where the second and third inequalities are applications of the triangle inequality.
For $0<\gamma_\ell<\gamma_u$, define
\begin{align*}
  &a(\varepsilon,j_1,j_2,j,\gamma_\ell,\gamma_u)
  = \max\left\{ \varepsilon 2^{-\max\left\{ (j_1-j)\gamma_u, (j_1-j)\gamma_\ell \right\}} -  2^{-\min\left\{ (j_2-j)\gamma_u, (j_2-j)\gamma_\ell \right\}}, 0 \right\}.
\end{align*}
If $\gamma_\ell\le\gamma\le\gamma_u$ and
$a(\varepsilon,j_1,j_2,j,\gamma_\ell,\gamma_u)>0$, then
$a(\varepsilon,j_1,j_2,j,\gamma_\ell,\gamma_u)
\le \frac{\varepsilon 2^{-j_1\gamma}-2^{-j_2\gamma}}{2^{-j\gamma}}$
so that, for any $f\in\FGN(\gamma,\varepsilon B,B)$,
\begin{equation}\label{a_delta_bias_bound_eq}
  \begin{aligned}
  &\sup_{x\in[0,1]}|K_jf(x)-f(x)|\le B 2^{-j\gamma}
    \le B\frac{\varepsilon 2^{-j_1\gamma}-2^{-j_2\gamma}}{a(\varepsilon,j_1,j_2,j,\gamma_\ell,\gamma_u)}
    \le \frac{\Delta(j_1,j_2;f)}{a(\varepsilon,j_1,j_2,j,\gamma_\ell,\gamma_u)}   
  \end{aligned}
\end{equation}
where the last inequality uses (\ref{delta_bound_eq}).

In Appendix \ref{gamma_ci_sec_append}, we provide an interval
$[\hat\gamma_\ell,\hat\gamma_u]$ that contains $\gamma$ on the event in
(\ref{delta_coverage_eq}).
Letting $\hat\jmath$, $\hat\jmath_1$ and $\hat\jmath_2$ be data dependent
values that are contained in $\mathcal{J}_n$ with probability one,
it follows from (\ref{a_delta_bias_bound_eq}) that, on the event that
(\ref{fxh_coverage_eq}) and (\ref{delta_coverage_eq}) both hold, the band
\begin{align*}%
\hat f(x,\hat\jmath)\pm \left[c(\hat\jmath)+\frac{\hat \Delta(\hat\jmath_1,\hat\jmath_2)+\tilde c(\hat\jmath_1,\hat\jmath_2)}{a(\varepsilon,\hat\jmath_1,\hat\jmath_2,\hat\jmath,\hat \gamma_\ell,\hat\gamma_u)}\right]
\end{align*}
contains $f(x)$ for all $x\in [0,1]$.
Since $\hat\jmath_1$, $\hat\jmath_2$ and $\hat\jmath$ can be data
dependent, we can simply choose them to
minimize
the length of this band.
For concreteness, we will assume that $\mathcal{J}_n$ is finite for each $n$, so
that a minimum is taken:
\begin{align*}%
c(\hat\jmath)+\frac{\hat \Delta(\hat\jmath_1,\hat\jmath_2)+\tilde c(\hat\jmath_1,\hat\jmath_2)}{a(\varepsilon,\hat\jmath_1,\hat\jmath_2,\hat\jmath,\hat \gamma_\ell,\hat\gamma_u)}
= \min_{j,j_1,j_2\in\mathcal{J}_n}\left[c(j)+\frac{\hat \Delta( j_1, j_2)+\tilde c(j_1, j_2)}{a(\varepsilon, j_1, j_2, j,\hat \gamma_\ell,\hat\gamma_u)}\right],
\end{align*}
where we use the convention that
$\frac{\hat \Delta( j_1, j_2)+\tilde c(j_1,j_2)}{a(\varepsilon, j_1, j_2, j,\hat \gamma_\ell,\hat\gamma_u)}$ is
equal to $+\infty$ if
$a(\varepsilon, j_1, j_2, j,\hat\gamma_\ell,\hat\gamma_u)=0$, so that the
minimum is only over $j,j_1,j_2$ such
that $a(\varepsilon, j_1, j_2, j,\hat \gamma_\ell,\hat\gamma_u)>0$.
The half-length of this band is then bounded by
\begin{align}\label{confidence_band_width_eq}
\min_{j,j_1,j_2\in\mathcal{J}_n}\left[
c(j)
+\frac{B(2^{-j_1\gamma}+2^{-j_2\gamma})+2\tilde c(j_1,j_2)}{a(\varepsilon,j_1,j_2,j,\hat \gamma_\ell,\hat \gamma_u)}
\right]
\end{align}
on the event that
(\ref{fxh_coverage_eq}) and (\ref{delta_coverage_eq}) both hold (here we use the
upper bound in (\ref{delta_bound_eq})).
In Appendix \ref{confidence_band_length_sec_append}, we use this bound to show that
this confidence band, constructed
with $\tilde\varepsilon=\varepsilon/\tilde C$ in place of $\varepsilon$,
satisfies the requirements of Theorem \ref{upper_bound_thm_general_kernel}.

\subsection{Alternative Definition of Self-Similarity Constant}\label{alternative_self_sim_sec}

We have defined $\Fselfsim(\gamma,B,\varepsilon)$ to be the class
of functions in $\FHol(\gamma,B)$ such that the lower bound in
(\ref{GN_j_condition}) holds with $b_1=\varepsilon B$.
Under (\ref{tildeCgamma_bound_eq}), this means that the self-similarity constant $\varepsilon$ gives
the ratio between the upper and lower bound on bias, up to the constant $\tilde C$.
The coverage condition takes the union of these classes with $\varepsilon$ fixed, so that large values of the H\"{o}lder constant require proportionally large values of the lower bound.

Alternatively, one could fix the lower bound $b_1=\varepsilon B$ when taking the
union of these classes.
This leads to the class
$\Fselfsimalt(\gamma,B,b_1)=\Fselfsim(\gamma,B,b_1/B)$.
Of course, this does not change the conclusion of Theorem
\ref{exponent_adaptation_thm_general_kernel} (adaptation to $\gamma$ with $B$ fixed) since
the formulation of this problem remains the same.
For adaptation to $B$, however, we obtain a different formulation,
with coverage required over the class
$\cup_{B\in [\underline B,\overline B]}\Fselfsimalt(\gamma,B,b_1)
=\Fselfsimalt(\gamma,\overline B,b_1)
=\Fselfsim(\gamma,\overline B,b_1/\overline B)$.
As the next theorem shows, this leads to a much more negative result: adaptation
to the H\"{o}lder constant is completely impossible.

\begin{theorem}\label{constant_adaptation_thm_alt_general_kernel}
  Let $\gamma>0$ and let
  $0<2\alpha<\beta<1$.  Let $K$ be a kernel satisfying
  (\ref{kernel_lower_bound_assump}).  There exists $\underline j_{K,\gamma}$,
  $C_{K,\gamma,*}>0$ and $\eta_{K,\gamma}>0$ such that, for any
  $0< B\le \overline B$,
  $b_1\le \eta_{K,\gamma}B$ and $\underline\ell\ge \underline j_{K,\gamma}$,
  \begin{align*}
    &R^*_{n,\alpha,\beta}(\Fselfsimalt(\gamma,B,b_1;K,\underline\ell), 
\Fselfsimalt(\gamma,\overline B,b_1;K,\underline \ell))  \\
    &\ge (1+o(1))C_{K,\gamma,*}\overline B^{1/(2\gamma+1)}\left(\sigma_n^2\log (1/\sigma_n)\right)^{\gamma/(2\gamma+1)}.
  \end{align*}
\end{theorem}

\subsection{Discussion}\label{discussion_sec}

The confidence band in Section \ref{achieving_bounds_sec_main} %
builds on the important work of
\cite{bull_honest_2012} and \cite{chernozhukov_anti-concentration_2014} in
constructing an upper bound on bias and using this to widen the confidence
interval
(see also \cite{knafl_model_1982,donoho_statistical_1994,armstrong_simple_2020}
for confidence intervals for $f$ at a point in the nonadaptive case).
In contrast to these papers, which
derive bounds on the bias of an estimator with bandwidth selected using Lepski's
method,
we bound the bias directly for each bandwidth and use the width of the resulting
confidence band to choose the bandwidth (note, however, that the two approaches are
related, since the bound on the bias ultimately comes from comparisons of
estimates at different bandwidths, either explicitly in our approach, or
implicitly through the use of Lepski's method to choose the bandwidth).
This makes it easier to derive explicit bounds, and it may be needed to get the
optimal form $C\varepsilon^{-1/(2\gamma+1)}$ of the adaptation penalty
(\cite{bull_honest_2012} and \cite{chernozhukov_anti-concentration_2014} show
that their procedures are adaptive up to a constant, but do not derive how this
constant depends on $\varepsilon$).

An alternative approach to ensuring coverage, used by
\cite{gine_confidence_2010}, is undersmoothing, which uses a bandwidth sequence
for which variance slightly dominates bias.  As noted by
\cite{bull_honest_2012} and \cite{chernozhukov_anti-concentration_2014}, this
leads to a slightly slower rate of convergence, so that the confidence band is
not fully adaptive.  
Our lower bounds shed some light on this question: one must
always pay an adaptation penalty of order $\varepsilon^{-1/(2\gamma+1)}$ when
$\varepsilon$ is fixed, which means that letting $\varepsilon=\varepsilon_n\to 0$ requires
paying a penalty in the rate.
In practice, however, 
for any given finite sample size $n$, one only achieves coverage over a class
$\Fselfsim$ corresponding to some $\varepsilon_n>0$; 
undersmoothed confidence bands choose such a sequence implicitly.
To make this transparent, one can explicitly specify $\varepsilon_n$, and report
a confidence band that is valid for the given self-similarity constant and noise
level, even if the ``asymptotic promise'' states that $\varepsilon_n\to 0$ (while our arguments do not formally cover the case where $\varepsilon=\varepsilon_n\to 0$, it appears that they could be extended to allow $\varepsilon_n\to 0$ at a slow enough rate).

There has been some discussion in the literature of whether or how
self-similarity conditions can lead to a practical approach to constructing
confidence bands.
If ``practical'' means that the confidence band should not
require the user to choose any regularity constants a priori, then our results
show that the answer is ``no.''  On the other hand, if one sees the
self-similarity constant as an interpretable object, then we need not be so
pessimistic.
Indeed, the confidence band we construct is ``practical'' in the
sense that it has valid coverage for a given noise level without relying on
conservative constants or sequences.

It is helpful to contrast the role of self-similarity conditions in our setting with
regularity conditions used to construct confidence intervals for the
mean of a univariate random variable.  To form a non-trivial confidence
interval for the mean of a univariate random variable, one must place some
conditions on the tails of the distribution (see \cite{bahadur_nonexistence_1956}).  One approach is to choose some $\delta>0$, and assume that
the $2+\delta$ moment is bounded by $1/\delta$.  Subject to this coverage
requirement, the optimal width of the confidence interval does not depend on
$\delta$ asymptotically: adding and subtracting the $1-\alpha/2$ quantile of a
normal distribution times the sample standard deviation leads to an
asymptotically valid confidence interval regardless of the particular choice
of $\delta>0$.  Thus, one can state that this confidence interval is
asymptotically valid and optimal under a bounded $2+\delta$ moment, without worrying about
the exact choice of $\delta$.  Our results show that this is not
the case with self-similarity constants: no single confidence band is
asymptotically valid and optimal under $\varepsilon$-self-similarity for all
$\varepsilon$.

\section{Proofs of Lower Bounds}\label{lower_bounds_proofs_sec}

This section proves
Theorems \ref{constant_adaptation_thm_general_kernel},
\ref{exponent_adaptation_thm_general_kernel} and
\ref{constant_adaptation_thm_alt_general_kernel}.
To prove these lower bounds, we proceed as follows.
Let $\widetilde{\mathcal{F}}(\gamma,B,a,b)$ denote the class of functions in
$\FHol(\gamma,B)$ supported on $[a,b]$:
\begin{align*}
  \widetilde{\mathcal{F}}(\gamma,B,a,b) = \{f\in\FHol(\gamma,B): f(t)=0 \text{ all } t\notin [a,b]\}.
\end{align*}
While functions in $\widetilde{\mathcal{F}}(\gamma,B,a,b)$ need not be
self-similar since this class does not impose a lower bound on bias, we can
ensure self-similarity by adding a function supported outside of
$[a,b]$ to this class, so long as this function satisfies the necessary upper
and lower bounds (after adjusting some constants).

Section \ref{general_lower_bound_sec} presents a lower bound for adaptation to
the singleton class $\{g\}$ for confidence bands with coverage under $g$ and
under the class $\{f\}+\widetilde{\mathcal{F}}(\gamma,B,a,b)$, for any functions
$f$ and $g$ supported outside of $[a,b]$.
Following standard arguments relating adaptive confidence sets to minimax
testing, such a bound follows so long as it is difficult to test
between $f$ and $g$ (which holds if $f$ and $g$ are close in $L_2$ norm), by
showing that it is difficult to test between $\{0\}$ (the zero function) and
functions in $\widetilde{\mathcal{F}}(\gamma,B,a,b)$ for which the supremum
over $[a,b]$ is sufficiently far from zero (which essentially follows from \cite{lepski_asymptotically_2000}).
Section \ref{constructing_functions_sec_new} constructs functions $g$ and $f$
such that the classes used in Section \ref{general_lower_bound_sec} satisfy the
self-similarity condition for appropriate $B$, $\gamma$ and $\varepsilon$, so
that the the lower bound in Section \ref{general_lower_bound_sec} can be used to
give bounds on adaptation between self-similarity classes.
For Theorems \ref{constant_adaptation_thm_general_kernel} and
\ref{constant_adaptation_thm_alt_general_kernel}, the functions $g$ and $f$ can
be taken to be equal, and the result follows almost immediately; Section
\ref{constant_adaptation_proof_sec} gives the necessary details to complete the
proofs.
To complete the proof of Theorem \ref{exponent_adaptation_thm_general_kernel},
we use the results in Section \ref{constructing_functions_sec_new} to construct
a function $g\in \Fselfsim(\gamma,1,\varepsilon')$ and a
sequence of functions $f_n$ converging to $g$ such that
$\{f_n\}+\widetilde{\mathcal{F}}(\gamma-\delta_n,1/2,a,b)\subseteq
\Fselfsim(\gamma-\delta_n,1,\varepsilon)$ where $\delta_n$ is a sequence
converging to zero.
Theorem \ref{exponent_adaptation_thm_general_kernel} then follows by using the lower bounds in Section \ref{general_lower_bound_sec} and choosing
the sequence $\delta_n$ to ensure that $f_n$ converges to $g$ quickly
enough, while making the testing problem for the class
$\widetilde{\mathcal{F}}(\gamma-\delta_n,1/2,a,b)$ sufficiently difficult.
These arguments are given in Section \ref{exponent_adaptation_proof_sec}.

\subsection{General Lower Bound}\label{general_lower_bound_sec}

In this section, we prove the following lower bound for adaptation between classes
of the form $\{g\}+\widetilde{\mathcal{F}}(\gamma,B,a,b)$.
For a function $f:\mathbb{R}\to\mathbb{R}$, let $\|f\|=\sqrt{\int f(t)^2\, dt}$
denote the $L_2$ norm of the function $f$.

\begin{lemma}\label{general_lower_bound_lemma}
  Let $a<b$ be given, and let $f_n$ and $g_n$ be sequences of functions with
  $f_n(t)=g_n(t)=0$ for $t\in [a,b]$.  Suppose $\|f_n-g_n\|/\sigma_n\to 0$.
  Let $0<\underline\gamma\le\overline\gamma$ be given, and let $\kappa$ be a
  function with finite support with $\kappa\in \FHol(\gamma,1)$ for all $\gamma\in (0,\overline\gamma]$.
  Let $B>0$ and let
  $C(\gamma,B,\kappa)=\left[\frac{4}{2\gamma+1}B^{1/\gamma}/\|\kappa\|^2\right]^{\frac{\gamma}{2\gamma+1}}\kappa(0)$.
  Then, for any sequence $\gamma_n\in [\underline\gamma,\overline\gamma]$ and
  any $0<2\alpha<\beta<1$,
  \begin{align*}
    R^*_{n,\alpha,\beta}&\left( \{g_n\}, \left\{ \{f_n\} + \widetilde{\mathcal{F}}(\gamma_n,B,a,b) \right\}\cup \{g_n\} \right)   \\
    &\ge C(\gamma_n,B,\kappa)\left(\sigma_n^2\log
   (1/\sigma_n)\right)^{\gamma_n/(2\gamma_n+1)}(1+o(1)).
  \end{align*}

\end{lemma}

To prove this result, we begin with a lemma relating $R^*_{n,\alpha,\beta}$ to minimax
bounds on statistical hypothesis tests.
For sets $\mathcal{F}$ and $\mathcal{G}$, let $\dtest(\mathcal{F},\mathcal{G})$ denote the maximum difference between minimax power and size of a test of $H_0:\mathcal{F}$ vs $H_1:\mathcal{G}$:
\begin{align*}
\dtest(\mathcal{F},\mathcal{G})
=\sup_{\phi} \inf_{f\in\mathcal{F},\, g\in\mathcal{G}} |E_g\phi(Y)-E_f\phi(Y)|
\end{align*}
where $E_f$ denotes expectation under the function $f$, and the supremum is over
all tests $\phi$ based on $Y$ observed at noise level $\sigma_n$ (i.e. all
measurable functions with range $[0,1]$).
The following lemma is essentially Lemma 6.1 in \cite{robins_adaptive_2006}, with the conclusion of the argument stated nonasymptotically.

\begin{lemma}\label{testing_ci_lemma_new}
Let $\alpha,\beta$ and $\tilde R$ be given and let $\mathcal{G}\subseteq \mathcal{F}$.  Suppose that
\begin{align*}
\text{for some }f_0\in\mathcal{G},\,
\dtest\left(\{f_0\},\mathcal{F}\cap \{f:\sup_{x\in[0,1]} |f(x)-f_0(x)|\ge \tilde R\}\right) 
<\beta-2\alpha.
\end{align*}
Then $R^*_{n,\alpha,\beta}(\mathcal{G},\mathcal{F})\ge R^*_{n,\alpha,\beta}(\{f_0\},\mathcal{F})\ge \tilde R$.
\end{lemma}
\begin{proof}
Suppose, to get a contradiction, that
$R^*_{n,\alpha,\beta}(\{f_0\},\mathcal{F})< \tilde R$.
Then there exists a confidence band $\mathcal{C}_n(\cdot)\in
\mathcal{I}_{n,\alpha,\mathcal{F}}$ with 
$R=R_\beta(\mathcal{C}_n;\{f_0\})=q_{\beta,f_0}\left(\sup_{x\in[0,1]}\length(\mathcal{C}_n(x))\right)<\tilde
R$, so that
\begin{equation}\label{length_quantile_eq}
  \begin{aligned}
    &P_{f_0}\left(\sup_{x\in[0,1]}\length\left(\mathcal{C}_n(x)\right)> R \right)  \\
    &= 1-P_{f_0}\left(\sup_{x\in[0,1]}\length\left(\mathcal{C}_n(x)\right)\le R \right)\le 1-\beta. 
  \end{aligned}
\end{equation}
Let us abuse notation slightly and let $\mathcal{C}_n$ denote the set of functions $f$ contained in the confidence band $\mathcal{C}_n(\cdot)$, so that $f\in\mathcal{C}_n$ iff. $f(t)\in \mathcal{C}_n(t)$ all $t\in [0,1]$.
Let $\phi=1$ if there exists a function $f$ satisfying $f\in\mathcal{F}\cap
\{f:\sup_{x\in[0,1]}|f(x)-f_0(x)|\ge \tilde R\}$ with $f\in \mathcal{C}_n$.
It is immediate from the definition of this test and the assumption that
$\mathcal{C}_n(\cdot)\in \mathcal{I}_{n,\alpha,\mathcal{F}}$ that  
\begin{align}\label{test_ci_lemma_power_eq}
 \inf_{f\in\mathcal{F}\cap \{f:\sup_{x\in[0,1]} |f(x)-f_0(x)|\ge \tilde R\}} E_f\phi\ge
1-\alpha 
\end{align}
 (i.e. the test has minimax power at least $1-\alpha$ for
$H_1:\mathcal{F}\cap \{f: \sup_{x\in[0,1]}|f(x)-f_0(x)|\ge \tilde R\}$). 

Now consider the level of the test for $H_0:\{f_0\}$.  We have
\begin{align*}
&E_{f_0}\phi(Y)=E_{f_0}\phi(Y)I(f_0\in\mathcal{C}_n)+E_{f_0}\phi(Y)I(f_0\notin\mathcal{C}_n) 
  \le E_{f_0}\phi(Y)I(f_0\in\mathcal{C}_n)+\alpha
\end{align*}
by the converage condition.  The event $\phi(Y)I(f_0\in\mathcal{C}_n)$ implies
that $\mathcal{C}_n$ contains both $f_0$ and a function $f_1$ with
$f_1\in\mathcal{F}$ and $\sup_{x\in[0,1]}|f_1(x)-f_0(x)|\ge \tilde R$.
This, in turn, implies that $\sup_{x\in[0,1]}\length(\mathcal{C}_n(x))\ge \tilde
R>R$ on this event so that, by (\ref{length_quantile_eq}), the probability of
this event under $f_0$ is bounded by $1-\beta$.  Thus, by the above display,
$E_{f_0}\phi(Y)\le 1-\beta+\alpha$.  Combining this with
(\ref{test_ci_lemma_power_eq}), it follows that $\inf_{f\in\mathcal{F}\cap \{f:\sup_{x\in[0,1]} |f(x)-f_0(x)|\ge \tilde R\}} E_f\phi-E_{f_0}\phi\ge
1-\alpha-1+\beta-\alpha=\beta-2\alpha$, which contradicts the assumptions of the theorem.
\end{proof}

To deal with minimax tests over classes that add functions $f_n$ and $g_n$, we
will also need the following lemma.

\begin{lemma}\label{F_plus_f0_testing_lemma_new}
For any functions $f_0$ and $g_0$ and sets $\mathcal{F}$ and $\mathcal{G}$,
\begin{align*}
&\dtest(\mathcal{F}+\{f_0\},\mathcal{G}+\{g_0\})
=\dtest(\mathcal{F},\mathcal{G}+\{g_0-f_0\})  \\
&\le \dtest(\mathcal{F},\mathcal{G}) + \sup_{\alpha}\left[\Phi\left(\|f_0-g_0\|/\sigma_n-z_{1-\alpha}\right)-\alpha\right]
\le \dtest(\mathcal{F},\mathcal{G}) + \|f_0-g_0\|/\sigma_n.
\end{align*}
\end{lemma}
\begin{proof}
The first equality follows since $f_0$ can be added or subtracted from $Y$ before performing any test, so that the supremum over tests $\phi(Y)$ is the same as the supremum over tests $\phi(Y-f_0)$.  For the first inequality, note that
\begin{align*}
\dtest(\mathcal{F},\mathcal{G}+\{g_0-f_0\})
=\sup_{\phi} \inf_{f\in\mathcal{F},\, g\in\mathcal{G}} |E_{g+f_0-g_0}\phi(Y)-E_f\phi(Y)|  \\
\le \sup_{\phi} \inf_{f\in\mathcal{F},\, g\in\mathcal{G}} \left[|E_{g+f_0-g_0}\phi(Y)-E_{g}\phi(Y)|+|E_{g}\phi(Y)-E_f\phi(Y)|\right].
\end{align*}
For any $g$, the first term is bounded by
$\sup_{\phi}|E_{g+f_0-g_0}\phi(Y)-E_{g}\phi(Y)|$
which, using the Neyman-Pearson lemma and some calculations (see Example 2.1 in
\cite{ingster_nonparametric_2003}), can be seen to be equal to
\begin{align*}
\sup_{\alpha}\left[\Phi\left(\|f_0-g_0\|/\sigma_n-z_{1-\alpha}\right)-\Phi(z_{1-\alpha})\right]
\le \|f_0-g_0\|/\sigma_n,
\end{align*}
where the inequality follows from Taylor's theorem, since the derivative of the standard normal cdf is bounded by $1/\sqrt{2\pi}\le 1$.
\end{proof}

With these results in hand, we can now complete the proof of Lemma \ref{general_lower_bound_lemma}.
Let $c_n=C(\gamma_n,B,\kappa)\left(\sigma_n^2\log
  (1/\sigma_n)\right)^{\gamma_n/(2\gamma_n+1)}$.  Given $\eta>0$, let
\begin{align*}
  \mathcal{H}_n&= \left\{ \{f_n\} + \widetilde{\mathcal{F}}(\gamma_n,B,a,b) \right\} \cap \{f: \sup_{x\in [0,1]} |f(x)-g_n(x)| \ge (1-\eta)c_n\}.
\end{align*}
By Lemma \ref{testing_ci_lemma_new}, the result will follow if we show that
$\dtest(\{g_n\},\mathcal{H}_n)\to 0$.
Furthermore, using the fact that $g_n$ and $f_n$ are supported outside $[a,b]$,
it follows that
$\{f_n\} + \widetilde{\mathcal{F}}(\gamma_n,B,a,b) \cap
\{f: \sup_{x\in [a,b]} |f(x)| \ge (1-\eta)c_n\} \subseteq \mathcal{H}_n$.  Since
taking a smaller set increases $\dtest$, it follows by 
Lemma \ref{F_plus_f0_testing_lemma_new}, that
$\dtest(\{g_n\},\mathcal{H}_n)$ is bounded by
\begin{align*}
  \dtest\left(\{0\},\widetilde{\mathcal{F}}(\gamma_n,B,a,b) \cap \{f: \sup_{x\in [a,b]}
 |f(x)| \ge (1-\eta) c_n \} \right) + \|f_n-g_n\|/\sigma_n.
\end{align*}
Since the second term converges to zero by assumption, it suffices to bound the
first term.

\begin{sloppypar}
To this end, we follow arguments on pp. 34-36 of \cite{lepski_asymptotically_2000}.
Let $A_\kappa$ be a bound on the support of $\kappa$ and let
\begin{align*}
&h_n=\left(\frac{(1-\eta)C({\gamma_n},B,\kappa)}{B\kappa(0)}\right)^{1/{\gamma_n}}\left(\sigma_n^2\log (1/\sigma_n)\right)^{1/(2{\gamma_n}+1)},  \\
&M_n=\left\lfloor\frac{b-a}{2 A_\kappa h_n}\right\rfloor-1,\quad
x_{n,k}=a+(2k-1)A_\kappa h_n,\quad k=1,\ldots, M_n  \\
&f_{k,n}(x)=B h_n^{\gamma_n} \kappa\left(\frac{x-x_{n,k}}{h_n}\right).
\end{align*}
By construction, the support of each $f_{k,n}$ is nonoverlapping and contained in $[a,b]$.
Also, the variance of $\int f_{k,n}(x)\, dY(x)$ is
\begin{align*}
B^2 h_n^{2{\gamma_n}} \int \kappa\left(\frac{x-x_{n,k}}{h_n}\right)^2\, dx
=B^2 h_n^{2{\gamma_n}+1} \int \kappa(u)^2\, d u=:s_n^2.
\end{align*}
Following arguments on pp. 35-36 of \cite{lepski_asymptotically_2000}, it will
then follow that
$\dtest(\{0\},\{f_{n,1},f_{n,2},\ldots,f_{n,M_n}\})\to 0$
so long as there exists $\delta>0$ such that, for large enough $n$,
$(s_n^2/\sigma_n^2)/(2\log M_n)\le (1-\delta)$.
Since each $f_{k,n}$ is contained in the set
$\widetilde{\mathcal{F}}(\gamma_n,B,a,b)\cap \{f:\sup_{x\in [a,b]}|f(x)|=(1-\eta)c_n\}$,
this will complete the proof.
\end{sloppypar}

For large enough $n$, we have $M_n\ge (b-a)/(3A_\kappa h_n)$ so that
\begin{align*}
2\log M_n&\ge 2\log h_n^{-1}+2\log [(b-a)/(3A_\kappa)]
  =\left(\frac{4}{2\gamma_n+1}+o(1)\right)\log (1/\sigma_n).
\end{align*}
We have
\begin{align*}
&\frac{s_n^2}{\sigma_n^2}=B^2 \|\kappa\|^2 h_n^{2{\gamma_n}+1}\sigma_n^{-2}
=B^2 \|\kappa\|^2
\left(\frac{(1-\eta)C({\gamma_n},B,\kappa)}{B\kappa(0)}\right)^{(2{\gamma_n}+1)/{\gamma_n}}\log (1/\sigma_n)  \\
&=(1-\eta)^{(2{\gamma_n}+1)/{\gamma_n}}\frac{4}{2{\gamma_n}+1} \log (1/\sigma_n).
\end{align*}
Thus, for $\delta$ smaller than a constant that depends only on $\overline
\gamma$ and $\underline \gamma$, we have, for $n$ large enough,
$(s_n^2/\sigma_n^2)/(2\log M_n)\le (1-\delta)$.

\subsection{Constructing Functions in Self-Similarity Classes}\label{constructing_functions_sec_new}

The main result of this section is to construct functions $g$ such that the class
$\{g\}+\widetilde{\mathcal{F}}(\gamma,B,a,b)$ satisfies the self-similarity
condition.
We first describe the construction, and then present the main lemma (Lemma
\ref{tilde_g_tilde_f_lemma_new}) showing self-similarity of these functions.
The remainder of this section is then devoted to the proof of Lemma
\ref{tilde_g_tilde_f_lemma_new}.

Let $\psi:\mathbb{R}\to\mathbb{R}$ be a function with $\|\psi\|=1$ with support
contained in $(-\psisupp,\psisupp)$ where $\psisupp<\infty$.  
Let $\psi_{\ell
  k}(x)=2^{\ell/2}\psi(2^\ell x - k)$.
We will consider functions that take the form
\begin{align}\label{f_tilde_beta_eq_new}
  f_{\{\tilde\beta\},\underline\ell}(x) = \sum_{\ell=\underline \ell}^\infty \tilde\beta_{\ell} \psi_{\ell k^*}(x),
\end{align}
for integers $k^*,\underline \ell$, chosen large enough to satisfy conditions
given below.
Given $0<\varepsilon<1$ and $0<\gamma-\delta\le
\gamma<\infty$, let
$\tilde f_{\underline \ell,\gamma,\delta,\varepsilon,1}$
be defined as in (\ref{f_tilde_beta_eq_new}) with
\begin{align*}
 \tilde \beta_{\ell}=\max \{2^{-\ell(\gamma+1/2)},\varepsilon 2^{-\ell(\gamma-\delta+1/2)}\}.
\end{align*}
Let $\tilde g_{\underline \ell,\gamma,1}$
be defined as in (\ref{f_tilde_beta_eq_new}) with
\begin{align*}
 \tilde \beta_{\ell}= 2^{-\ell(\gamma+1/2)}.
\end{align*}
Let $\tilde f_{\underline \ell,\gamma,\delta,\varepsilon,A}(x)=A \tilde
f_{\underline \ell,\gamma,\delta,\varepsilon,1}(x)$ and let
$\tilde g_{\underline \ell,\gamma,A}(x)=A\tilde g_{\underline \ell,\gamma,1}(x)$.

To get some intuition for this construction, note that,
if $\psi$ is a mother wavelet for some wavelet basis, then a function
constructed in this way has $\ell, k$th
wavelet coefficient given by $\tilde\beta_{\ell}$ for $\ell \ge \underline\ell$
and $k=k^*$ and $\ell, k$th wavelet coefficent $0$ for all other $\ell, k$.  If
the kernel $K$ in the self-similarity condition is the wavelet projection kernel
for this basis, self-similarity of $\tilde g_{\underline \ell,\gamma,A}$ and $\tilde f_{\underline \ell,\gamma,\delta,\varepsilon,A}$
would follow from standard calculations.
However, relying on such arguments would rule out convolution kernels, and would
also present an issue for nonsmooth wavelets (since we impose a
H\"{o}lder condition in addition to the bounds on bias).

We now present the main result of this section, showing that, if $k^*$ and
$\underline \ell$ are chosen appropriately, adding $\tilde g_{\underline
  \ell,\gamma,A}$ and $\tilde f_{\underline \ell,\gamma,\delta,\varepsilon,A}$
to functions in the classes $\widetilde{\mathcal{F}}(\gamma,B,a,b)$ yields self-similar
functions.
Let
$\underline C_{K,\psi}=\sup_{x\in\mathbb{R}}|K_0\psi(x)-\psi(x)|>0$.
Let $\|f\|_\infty=\sup_{t\in\mathbb{R}}|f(t)|$ denote the $L_\infty$ norm, and let
$\overline C_{K,\psi,\gamma}= 2 \|\psi^{(\lfloor \gamma \rfloor+1)}\|_{\infty}
(2C_{\psi})^{1-(\gamma-\lfloor \gamma \rfloor)}$.
Note that $\psi$ can be chosen so that $\overline C_{K,\psi,\gamma}$ is bounded
from above over $\gamma\le\overline\gamma$, and so that $\underline C_{K,\psi}>0$.

\begin{lemma}\label{tilde_g_tilde_f_lemma_new}
  Let $0<a<b$, $A>0$ and $\tilde B\ge 0$ be given, and let $K$ be a kernel that satisfies
  (\ref{kernel_lower_bound_assump}).
  Let $k^*>4(\psisupp+\Ksupp)$, and let $\underline\ell$ be large enough so that
  $2^{-\underline\ell}(k^*+\psisupp+\Ksupp)<a$.
  Then,
  for any $A^*\ge \overline C_{K,\psi,\gamma}A+\tilde B$ and $\varepsilon^*\le \underline C_{K,\psi}A/A^*$,
  \begin{align*}
    \widetilde{\mathcal{F}}(\gamma,\tilde B,a,b) + \{\tilde g_{\underline \ell,\gamma,A}\}
      \subseteq\Fselfsim\left( \gamma, A^*, \varepsilon^*; K, \underline\ell \right).
  \end{align*}
  For any $A^*\ge \overline C_{K,\psi,\gamma-\delta}A+\tilde B$,
  $0<\delta<\gamma$
  and $\varepsilon^*\le \tilde\varepsilon \underline C_{K,\psi}A/A^*$,
  \begin{align*}
    \widetilde{\mathcal{F}}(\gamma-\delta,\tilde B,a,b) + \{\tilde f_{\underline \ell,\gamma,\delta,\tilde\varepsilon,A}\}
    \subseteq \Fselfsim\left( \gamma-\delta, A^*, \varepsilon^*; K, \underline\ell \right).
  \end{align*}
\end{lemma}

To prove Lemma \ref{tilde_g_tilde_f_lemma_new}, we first note some conditions on
the support of the functions $\psi_{\ell k^*}$ and their projections.

\begin{lemma}\label{support_lemma}
  If the support of a function $f$ is contained in $(c,d)$ for some $c,d$, then the support of
  $K_jf$ is contained in $(c-2^{-j}C_K,d+2^{-j}C_K)$.
  In particular, letting $\tilde S_{j \ell}=(2^{-\ell}k^*-2^{-\ell}\psisupp-2^{-j}\Ksupp,2^{-\ell}k^*+2^{-\ell}\psisupp+2^{-j}\Ksupp)$
  the support of $K_j\psi_{\ell k^*}$ is
  contained in $\tilde S_{j \ell}$, and the support of $\psi_{\ell k^*}$ is
  contained in $\tilde S_{\ell \ell}$.
  Furthermore, if $k^*>4(\psisupp+\Ksupp)$, then $\tilde S_{j j}\cap \tilde S_{j
  \ell}=\emptyset$ for $\ell \ne j$.
\end{lemma}
\begin{proof}
  The first statement is immediate from the fact that
  $K_j(y,x)=2^jK(2^jy,2^jx)=0$ whenever $|x-y|>2^{-j}C_K$.
  The second statement then follows since the support of
  $\psi_{\ell k^*}$ is contained in $(2^{-\ell}k^*-2^{-\ell}\psisupp,2^{-\ell}k^*+2^{-\ell}\psisupp)$
  by the support condition on $\psi$.
  To verify the last statement, note that,
  for any $\ell\ge j+1$, elements in $\tilde S_{j \ell}$ are less than
  $2^{-j-1}k^*+2^{-j-1}\psisupp+2^{-j}\Ksupp$,
  which is less than $2^{-j}k^*-2^{-j}\psisupp-2^{-j}\Ksupp$
  (the lower support point of $\tilde S_{j j}$) so long as
  $k^*>3\psisupp+4\Ksupp$, which is guaranteed by the condition
  $k^*>4(\psisupp+\Ksupp)$.
  For any $\ell\le j-1$, elements in $\tilde S_{j, \ell}$ are greater than
  $2^{-j+1}k^*-2^{-j+1}\psisupp-2^{-j}\Ksupp$, which is greater than
  $2^{-j}k^*+2^{-j}\psisupp+2^{-j}\Ksupp$
  (the upper support point of $\tilde S_{j j}$) so long as
  $k^*>3\psisupp+2\Ksupp$, which is guaranteed by the condition
  $k^*>4(\psisupp+\Ksupp)$.
\end{proof}

We now use this to obtain a lower bound on projection bias.

\begin{lemma}\label{bias_lower_bound_lemma_new}
  Suppose that $K(y,x)$ satisfies (\ref{kernel_lower_bound_assump}), and let $f_{\{\tilde\beta\},\underline\ell}$
  be defined as in (\ref{f_tilde_beta_eq_new}), with $k^*>4(\psisupp+\Ksupp)$.
  Let $f^*$ be a function supported on the set
  $(2^{-\underline\ell}(k^*+\psisupp+2\Ksupp),\infty)$, and let $f=f_{\{\tilde\beta\},\underline\ell}+f^*$. 
  Then, for $j\ge \underline\ell$,
  \begin{align*}
    \sup_{x\in[0,2^{-j}(k^*+\psisupp+\Ksupp)]}|K_jf(x)-f(x)|\ge |\tilde \beta_j|\cdot 2^{j/2}\sup_{x\in\mathbb{R}}|K_0\psi(x)-\psi(x)|.
  \end{align*}
\end{lemma}
\begin{proof}
It follows from Lemma \ref{support_lemma} that, for $x\in \tilde S_{j j}$, we
have $f(x)=\psi_{j k^*}(x)$ and $K_jf(x)=K_j\psi_{j k^*}(x)$, so that
\begin{align*}
  &\sup_{x\in [0,2^{-j}(k^*+\psisupp+\Ksupp)]} |K_j f(x)-f(x)|
  \ge \sup_{x\in \tilde S_{j j}}|K_j f(x)-f(x)|  \\
  &= |\tilde\beta_j| \sup_{x\in \mathbb{R}}|K_j\psi_{j k^*}(x)-\psi_{j k^*}(x)|
    = |\tilde\beta_j| \cdot 2^{j/2} \sup_{x\in \mathbb{R}}|K_0\psi(x)-\psi(x)|
\end{align*}
where the last step follows by using a change of variables to note that
$K_j\psi_{j k^*}(x)-\psi_{j k^*}(x)
    = 2^{j/2}\left[ K_0\psi(u-k^*) - \psi(u-k^*) \right]$.
\end{proof}

Next, we obtain a H\"{o}lder condition on functions of the form given in
(\ref{f_tilde_beta_eq_new}) using the rate of decay of the coefficients $\tilde\beta_\ell$.

\begin{lemma}\label{wavelet_expansion_holder_lemma_new}
Let $\gamma>0$ and suppose
that $\psi$ is $\lfloor \gamma \rfloor+1$ times
differentiable.
Let $A$ be given and let $f(x)=f_{\{\tilde\beta\},\underline\ell}(x)$ be given
by (\ref{f_tilde_beta_eq_new}) where $|\tilde\beta_\ell| \le A 2^{-\ell(\gamma+1/2)}$
for all $\ell$.
Then $f\in\FHol(\gamma,2 A \|\psi^{(\lfloor \gamma \rfloor+1)}\|_{\infty} (2C_{\psi})^{1-(\gamma-\lfloor \gamma \rfloor)})$.  
\end{lemma}
\begin{proof}
Since the supports of the functions $\psi_{\ell k^*}$ do not overlap with each
other by Lemma \ref{support_lemma}, it follows from
Lemma \ref{holder_sum_lemma_new} below that it suffices to show that $x\mapsto
\tilde\beta_\ell \psi_{\ell k^*}(x)$ is in $\FHol(\gamma, A \|\psi^{(\lfloor \gamma \rfloor)+1}\|_{\infty}
(2C_{\psi})^{1-(\gamma-\lfloor \gamma \rfloor)})$ for each $\ell$.
Given $\ell$, let $x$ and $x'$ be in the support of $\psi_{\ell k^*}$ so that
$x,x'\in [2^{-\ell}k^* - 2^{-\ell}\psisupp, 2^{-\ell}k^* + 2^{-\ell}\psisupp]$.
Then
\begin{align*}
  &\left| \tilde\beta_{\ell}\psi_{\ell k^*}^{(\lfloor \gamma \rfloor)}(x) - \tilde\beta_{\ell}\psi_{\ell k^*}^{(\lfloor \gamma \rfloor)}(x') \right|
  = |\tilde\beta_{\ell}|2^{\ell(\lfloor \gamma \rfloor + 1/2)}\left| \psi^{(\lfloor \gamma \rfloor)}(2^\ell x + k) - \psi^{(\lfloor \gamma \rfloor)}(2^\ell x' + k) \right|  \\
  &\le \|\psi^{(\lfloor \gamma \rfloor+1)}\|_{\infty}\cdot |\tilde\beta_{\ell}|2^{\ell(\lfloor \gamma \rfloor + 1/2)} \cdot 2^{\ell}|x-x'|  \\
  &= \|\psi^{(\lfloor \gamma \rfloor+1)}\|_{\infty}\cdot |\tilde\beta_{\ell}|2^{\ell(\lfloor \gamma \rfloor + 1/2)}\cdot (2C_{\psi}) \cdot (2C_{\psi})^{-1}2^{\ell}|x-x'|  \\
  &\le \|\psi^{(\lfloor \gamma \rfloor+1)}\|_{\infty}\cdot |\tilde\beta_{\ell}|2^{\ell(\lfloor \gamma \rfloor + 1/2)}\cdot (2C_{\psi}) \cdot (2C_{\psi})^{-(\gamma-\lfloor \gamma \rfloor)}2^{\ell(\gamma-\lfloor \gamma \rfloor)}|x-x'|^{\gamma-\lfloor \gamma \rfloor} 
\end{align*}
where the last inequality uses the fact that $(2C_{\psi})^{-1}2^{\ell}|x-x'|\le
1$ by the conditions on $x,x'$. 
If $|\tilde \beta_{\ell}|\le A 2^{-\ell(\gamma+1/2)}$, then this is bounded by
$A \|\psi^{(\lfloor \gamma \rfloor+1)}\|_{\infty} (2C_{\psi})^{1-(\gamma-\lfloor \gamma \rfloor)}|x-x'|^{\gamma-\lfloor \gamma \rfloor}$
as required.
\end{proof}

We have used the following lemma.

\begin{lemma}\label{holder_sum_lemma_new}
  Let $\{g_k\}_{k=1}^\infty$ be a sequence of functions with nonoveralapping support with
  $g_k\in \FHol(\gamma,B)$ for each $k$.  Let $f=\sum_{k=1}^\infty g_k$.  Then
  $f\in \FHol(\gamma, 2B)$.
\end{lemma}
\begin{proof}
Let $x,x'$ be given.  We need to show that $|f^{\lfloor \gamma
  \rfloor}(x)-f^{\lfloor \gamma \rfloor}(x')|\le 2B |x-x'|^{\gamma - \lfloor
  \gamma \rfloor}$.  If $x$ and $x'$ are both in the support of $g_k$ for some
$k$, or if $x$ and $x'$ are not in the support of $g_k$ for any $k$, then this
follows immediately.  If $x$ is in the support of $g_k$ and $x'$ is in the
support of $g_{k'}$ for some $k'\ne k$, let $\overline x$ denote the upper
endpoint of the support of $g_k$ and let $\underline x'$ denote the lower endpoint of
the support of $g_{k'}$, and assume without loss of generality that $\overline
x\le \underline x'$.  By the H\"{o}lder condition on $g_k$ and $g_{k'}$, we have
$g_k^{\lfloor \gamma \rfloor}(\overline x)=g_{k'}^{\lfloor \gamma
  \rfloor}(\underline x')=0$, so that
$|f^{\lfloor \gamma
  \rfloor}(x)-f^{\lfloor \gamma \rfloor}(x')|
=|g_k^{\lfloor \gamma \rfloor}(x) - g_k^{\lfloor \gamma \rfloor}(\overline x)
+ g_{k'}^{\lfloor \gamma \rfloor}(x) - g_{k'}^{\lfloor \gamma
  \rfloor}(\underline x')|
\le B|x-\overline x|^{\gamma - \lfloor \gamma \rfloor}+B|x'-\underline
x'|^{\gamma - \lfloor \gamma \rfloor}
\le 2B |x-x'|^{\gamma - \lfloor \gamma \rfloor}$.
Finally, if $x$ is in the support of some $g_k$ and $x'$ is not in the support
of $g_{k'}$ for any $k'$, then, letting $[\underline x,\overline x]$ denote the
support of $g_k$, 
$|f^{\lfloor \gamma
  \rfloor}(x)-f^{\lfloor \gamma \rfloor}(x')|
= |g_k^{\lfloor \gamma
  \rfloor}(x)|
\le B \min\{|x-\underline x|^{\gamma - \lfloor \gamma \rfloor}, |x-\overline
x|^{\gamma - \lfloor \gamma \rfloor}\}
\le B |x-x'|^{\gamma - \lfloor \gamma \rfloor}$.
\end{proof}

With these results in hand, we can now prove Lemma
\ref{tilde_g_tilde_f_lemma_new}.  
Let $f^*\in\widetilde{\mathcal{F}}(\gamma,\tilde B,a,b)$
and let $g=\tilde g_{\underline \ell,\gamma,A}+f^*$
and $f=\tilde f_{\underline \ell,\gamma,\delta,\tilde\varepsilon,A}+f^*$.
It follows from Lemma \ref{wavelet_expansion_holder_lemma_new} that
$\tilde g_{\underline \ell,\gamma,A}\in \FHol(\gamma,\overline
C_{K,\psi,\gamma}A)$
and 
$\tilde f_{\underline \ell,\gamma,\delta,\tilde\varepsilon,A}\in
\FHol(\gamma-\delta,\overline C_{K,\psi,\gamma-\delta}A)$.
Thus,
$g\in \FHol(\gamma,\overline
C_{K,\psi,\gamma}A+\tilde B)\subseteq \FHol(\gamma,A^*)$
for $A^*\ge \overline C_{K,\psi,\gamma}A+\tilde B$
and
$f\in \FHol(\gamma-\delta,\overline C_{K,\psi,\gamma-\delta}A+\tilde B)\subseteq
\FHol(\gamma-\delta, A^*)$
for
$A^*\ge \overline C_{K,\psi,\gamma-\delta}A+\tilde B$.
To verify the lower bound on bias,
note that, for $j\ge \underline \ell$, we
have, by 
Lemma \ref{bias_lower_bound_lemma_new},
$\sup_{x\in[0,1]}|K_jg(x)-g(x)|\ge A 2^{-j(\gamma+1/2)}\cdot
2^{j/2}\underline C_{K,\psi}=A 2^{-j\gamma}\underline
C_{K,\psi}=(\underline C_{K,\psi}A/A^*)\cdot A^*\cdot 2^{-j\gamma}$.
Thus, for
$A^*\ge \overline C_{K,\psi,\gamma}A+\tilde B$
and
$\varepsilon^*\le \underline C_{K,\psi}A/A^*$, we have 
$g\in \Fselfsim(\gamma, A^*, \varepsilon^*; K, \underline \ell)$
as required.
Similarly,
$\sup_{x\in[0,1]}|K_jf(x)-f(x)|\ge \tilde\varepsilon A 2^{-j(\gamma-\delta+1/2)}\cdot
2^{j/2}\underline C_{K,\psi}=\tilde\varepsilon A 2^{-j(\gamma-\delta)}\underline
C_{K,\psi}=\tilde\varepsilon (\underline C_{K,\psi}A/A^*)\cdot A^*\cdot
2^{-j(\gamma-\delta)}$,
so that, for
$A^*\ge \overline C_{K,\psi,\gamma-\delta}A+\tilde B$
and
$\varepsilon^*\le \tilde\varepsilon \underline C_{K,\psi}A/A^*$,
we have
$f\in \Fselfsim(\gamma-\delta, A^*,\varepsilon^*; K,\underline\ell)$
as required.

\subsection{Proofs of Theorems \ref{constant_adaptation_thm_general_kernel} and
  \ref{constant_adaptation_thm_alt_general_kernel}}\label{constant_adaptation_proof_sec}

To prove Theorem \ref{constant_adaptation_thm_general_kernel},
let $\tilde g_{\underline\ell,\gamma,A}$ be defined as in Section \ref{constructing_functions_sec_new} with
$k^*$ and $\underline\ell$ chosen so that $k^*>4(\psisupp+\Ksupp)$ and $2^{-\underline\ell}(k^*+\psisupp+\Ksupp)<1/2$,
and with $A=B/(2\max\{\overline C_{K,\psi,\gamma},1\})$.  By Lemma
\ref{tilde_g_tilde_f_lemma_new},
$\tilde g_{\underline\ell,\gamma,A}\in \Fselfsim(\gamma,B,\varepsilon';K,\underline\ell)$
so long as $\varepsilon'\le \underline C_{K,\psi}/(2\max\{\overline C_{K,\psi,\gamma},1\})$.
Let $\tilde B=\min\{\tilde\varepsilon^{-1}B,\overline B\}-\overline
C_{K,\psi,\gamma}A$ where $\tilde\varepsilon=2\varepsilon \max\{\overline
C_{K,\psi,\gamma},1\}/\underline C_{K,\psi}$.
Applying Lemma \ref{tilde_g_tilde_f_lemma_new} with $\min\{\tilde\varepsilon^{-1}B,\overline
B\}$ playing the role of $A^*$, we have
$\widetilde{\mathcal{F}}(\gamma,\tilde B,1/2,1)+\{\tilde
g_{\underline\ell,\gamma,A}\}\subseteq \Fselfsim(\gamma,\min\{\tilde\varepsilon^{-1}B,\overline B\},\varepsilon;K,\underline\ell)$,
where we use the fact that the choice of $\tilde\varepsilon$ guarantees
$\underline C_{K,\psi}A/A^* \ge \varepsilon$.
If $\eta_{K,\gamma}$ is small enough, then we will have
$\min\{\tilde\varepsilon^{-1}B,\overline B\}\in [\underline B,\overline B]$, so
that this implies $\widetilde{\mathcal{F}}(\gamma,\tilde B,1/2,1)+\{\tilde
g_{\underline\ell,\gamma,A}\}\subseteq \cup_{B'\in[\underline B,\overline B]}
\Fselfsim(\gamma,B',\varepsilon; K,\underline\ell)$.
Applying Lemma \ref{general_lower_bound_lemma}, it follows that
$R^*_{n,\alpha,\beta}(\Fselfsim(\gamma,B,\varepsilon';K,\underline\ell), 
\cup_{B'\in[\underline B,\overline
  B]}\Fselfsim(\gamma,B',\varepsilon;K,\underline \ell))$
is bounded from below by 
$(1+o(1))\tilde B^{1/(2\gamma+1)}\left(\sigma_n^2\log
  (1/\sigma_n)\right)^{\gamma/(2\gamma+1)}$
times a term that depends only on $\gamma$.
The result follows by noting that, if $\eta_{K,\gamma}$ is chosen small enough,
then $\tilde B$ is bounded from below by a constant times
$\min\{\varepsilon^{-1}B,\overline B\}$, where the constant depends only on
$\underline C_{K,\psi}$ and $\overline C_{K,\psi,\gamma}$.

To prove Theorem \ref{constant_adaptation_thm_alt_general_kernel},
we use similar arguments with the same function $\tilde
g_{\underline\ell,\gamma,A}$ (defined with $k^*$ and $\underline\ell$ chosen so that $k^*>4(\psisupp+\Ksupp)$ and $2^{-\underline\ell}(k^*+\psisupp+\Ksupp)<1/2$,
and with $A=B/(2\max\{\overline C_{K,\psi,\gamma},1\})$).
By Lemma
\ref{tilde_g_tilde_f_lemma_new},
$\tilde g_{\underline\ell,\gamma,A}\in
\Fselfsim(\gamma,B,b_1/B;K,\underline\ell)=\Fselfsimalt(\gamma,B,b_1;K,\underline\ell)$
so long as
$b_1/B\le \underline C_{K,\psi}/(2\max\{\overline C_{K,\psi,\gamma},1\})$.
Let 
$\tilde B=\overline B-\overline C_{K,\psi,\gamma}A=\overline B- B\overline
C_{K,\psi,\gamma}/( 2\max\{ \overline C_{K,\psi,\gamma}, 1 \} )$.
Applying Lemma \ref{tilde_g_tilde_f_lemma_new} with $\overline B$ playing the role of $A^*$, we have
$\widetilde{\mathcal{F}}(\gamma,\tilde B,1/2,1)+\{\tilde
g_{\underline\ell,\gamma,A}\}\subseteq \Fselfsim(\gamma,\overline B,b_1/\overline B;K,\underline\ell)=\Fselfsimalt(\gamma,\overline B,b_1;K,\underline\ell)$,
so long as $b_1\le \underline C_{K,\psi} A=\underline C_{K,\psi} B/(2\max\{\overline C_{K,\psi,\gamma},1\})$.
The result follows by applying Lemma \ref{general_lower_bound_lemma} and noting
that $\tilde B\ge \overline B/2$.

\subsection{Proof of Theorem \ref{exponent_adaptation_thm_general_kernel}}\label{exponent_adaptation_proof_sec}

To prove Theorem \ref{exponent_adaptation_thm_general_kernel},
let
$\overline C=\sup_{\gamma'\in(0,\overline\gamma]} \overline C_{K,\psi,\gamma'}$
and let
$A=1/(2\overline C)$
and
$\tilde\varepsilon=2\varepsilon \overline C/\underline C_{K,\psi}$.
Let
$\tilde g_{\underline\ell,\gamma,A}$
and
$\tilde
f_{\underline\ell,\gamma,\delta,\tilde\varepsilon,A}$
be defined as in
Section \ref{constructing_functions_sec_new} with
$k^*$ and $\underline\ell$ chosen so that $k^*>4(\psisupp+\Ksupp)$ and
$2^{-\underline\ell}(k^*+\psisupp+\Ksupp)<1/2$.
By Lemma
\ref{tilde_g_tilde_f_lemma_new},
we have
$\tilde
g_{\underline\ell,\gamma,A}\in\Fselfsim(\gamma,1,\varepsilon')\subseteq \Fselfsim(\gamma,1,\varepsilon)$ for any
$\varepsilon\le \varepsilon'\le \underline C_{K,\psi}/(2 \overline C)$
and
$\widetilde{\mathcal{F}}(\gamma-\delta,1/2,1/2,1)+
\{\tilde f_{\underline\ell,\gamma,\delta,\tilde\varepsilon,A}\}
\subseteq \Fselfsim(\gamma-\delta,1,\varepsilon)$.
Thus, applying Lemma \ref{general_lower_bound_lemma}, we have, for any positive
sequence $\delta_n\to 0$,
\begin{align*}%
R^*_{n,\alpha,\beta}&\left( \Fselfsim(\gamma,1,\varepsilon'), \cup_{\gamma'\in[\underline\gamma,\overline\gamma]}\Fselfsim(\gamma',1,\varepsilon) \right)   \\
    &\ge C(\gamma-\delta_n,1/2,\kappa)\left(\sigma_n^2\log
   (1/\sigma_n)\right)^{(\gamma-\delta_n)/(2(\gamma-\delta_n)+1)}(1+o(1)).
\end{align*}
so long as
\begin{align}\label{f_g_to_zero_eq}
  \|\tilde g_{\underline\ell,\gamma,A} - \tilde
f_{\underline\ell,\gamma,\delta_n,\tilde\varepsilon,A}\|/\sigma_n\to 0.
\end{align}
Since $C(\gamma-\delta_n,1/2,\kappa)$ is bounded from below by a positive
constant that depends only on $\overline\gamma$, it suffices to find a sequence
$\delta_n\to 0$ such that (\ref{f_g_to_zero_eq}) holds and
\begin{align}\label{delta_n_exponent_limit_eq}
  \liminf_{n\to\infty} \frac{\left(\sigma_n^2\log
   (1/\sigma_n)\right)^{(\gamma-\delta_n)/(2(\gamma-\delta_n)+1)}}{\left(\sigma_n^2\log
  (1/\sigma_n)\right)^{\gamma/(2\gamma+1)}}\ge c \cdot \varepsilon^{-1/(2\gamma+1)}
\end{align}
for some constant $c$ that depends only on $\overline \gamma$ and $K$.

Let $\delta_n=C_n/\log n$ where
$C_n=(1-b_n)(2\gamma+1)\log\tilde\varepsilon^{-1}$ with $b_n=1/(\log n)^{1/2}$.
First, note that 
$\|\tilde g_{\underline\ell,\gamma,A} - \tilde
f_{\underline\ell,\gamma,\delta,\tilde\varepsilon,A}\|^2$ is equal to $A^2$
times
\begin{align*}
\sum_{\ell=\tilde\ell}^\infty
 \left( \tilde\varepsilon 2^{-\ell(\gamma-\delta+1/2)} - 2^{-\ell(\gamma+1/2)} \right)^2
  = \sum_{\ell=\tilde\ell}^\infty
 2^{-\ell(2\gamma+1)}\left( \tilde\varepsilon 2^{\ell\delta} - 1 \right)^2
\end{align*}
where $\tilde\ell=\tilde\ell(\tilde\varepsilon,\delta)$ is the minimum value of $\ell\ge \underline\ell$ such that
$\tilde\varepsilon 2^{\ell\delta}> 1$
(here we use the fact that the support of $\psi_{\ell k^*}$ does not overlap
with the support of $\psi_{\ell' k^*}$ for $\ell\ne\ell'$ by Lemma \ref{support_lemma}).
The above display is bounded by
\begin{align*}
  \tilde\varepsilon^2 \sum_{\ell=\tilde\ell}^{\infty} 2^{-\ell(2(\gamma-\delta)+1)}
  = \tilde\varepsilon^2 \sum_{\ell=0}^{\infty} 2^{-(\ell+\tilde\ell)(2(\gamma-\delta)+1)}
  = \tilde\varepsilon ^22^{-\tilde\ell(2(\gamma-\delta)+1)}\sum_{\ell=0}^{\infty} 2^{-\ell(2(\gamma-\delta)+1)}.
\end{align*}
Note that $2^{-\tilde \ell}<\tilde\varepsilon^{1/\delta}$, so
$2^{-\tilde\ell(2(\gamma-\delta)+1)}< \tilde\varepsilon^{(2(\gamma-\delta)+1)/\delta}$.
From this and the fact that
$\sum_{\ell=0}^{\infty} 2^{-\ell(2(\gamma-\delta)+1)}\le \sum_{\ell=0}^{\infty}
2^{-\ell}=2$, it follows that the above display is bounded by
$2\tilde\varepsilon^{2+(2(\gamma-\delta)+1)/\delta}=2\tilde\varepsilon^{(2\gamma+1)/\delta}$.
Plugging in $\delta_n=C_n/\log n$, dividing by $\sigma_n^2$ and taking logs
gives
\begin{align*}
  &\log \left[ \| \tilde f_{\underline
 \ell,\gamma,\delta_n,\tilde\varepsilon,A} - \tilde g_{\underline
  \ell,\gamma,A} \|^2/\sigma_n^2 \right]
  \le \frac{2\gamma+1}{\delta_n}\log\tilde\varepsilon+\log 2-\log (\sigma^2/n)+\log A^2  \\
  &=\left( \frac{(2\gamma+1)\log\tilde\varepsilon}{C_n}+1 \right)\log n+\log (2A^2/\sigma^2)
   =\frac{-b_n}{1-b_n}\log n+\log (2A^2/\sigma^2)
\end{align*}
which diverges to $-\infty$, so that exponentiating gives a sequence that
converges to $0$.  Thus, (\ref{f_g_to_zero_eq}) holds for this sequence $\delta_n$.

To verify (\ref{delta_n_exponent_limit_eq}) for this sequence $\delta_n$, note
that
\begin{align*}
&\frac{\gamma-\delta_n}{2(\gamma-\delta_n)+1}
-\frac{\gamma}{2\gamma+1}
=-\frac{\delta_n}{[2(\gamma-\delta_n)+1](2\gamma+1)}
=-\frac{\delta_n}{(2\gamma+1)^2}(1+o(1)).
\end{align*}
Thus,
\begin{align*}
&(\sigma_n^2)^{\frac{\gamma-\delta_n}{2(\gamma-\delta_n)+1}
-\frac{\gamma}{2\gamma+1}}
=(\sigma_n^2)^{-\frac{\delta_n}{(2\gamma+1)^2}(1+o(1))}
=(1+o(1))
n^{\frac{\delta_n}{(2\gamma+1)^2}(1+o(1))}  \\
&=\exp\left(
\frac{\delta_n}{(2\gamma+1)^2}(1+o(1))\log n
\right).
\end{align*}
Since $\delta_n\log n\to (2\gamma+1)\log \tilde\varepsilon^{-1}$, this converges
to $\exp\left(\frac{(2\gamma+1)\log \tilde\varepsilon^{-1}}{(2\gamma+1)^2}\right)=\tilde\varepsilon^{-1/(2\gamma+1)}$.
For the other term in (\ref{delta_n_exponent_limit_eq}), we have
\begin{align*}
&[\log (1/\sigma_n)]^{\frac{\gamma-\delta_n}{2(\gamma-\delta_n)+1}
-\frac{\gamma}{2\gamma+1}}
  =[\log \sigma^{-1}+(1/2)\log n]^{\mathcal{O}(1/\log n)}  \\
&=\exp\left(\mathcal{O}(1/\log n)\log [\log \sigma^{-1}+(1/2)\log n]\right)
\end{align*}
which converges to one as $n\to\infty$.
Thus, for this sequence $\delta_n$, the left hand side of
(\ref{delta_n_exponent_limit_eq}) converges to
$\tilde\varepsilon^{-1/(2\gamma+1)}=(2\overline C/\underline
C_{K,\psi})^{-1/(2\gamma+1)}\varepsilon^{-1/(2\gamma+1)}$.
Since $(2\overline C/\underline C_{K,\psi})^{-1/(2\gamma+1)}$ is bounded from
below by a positive constant uniformly over $\gamma\le \overline\gamma$, it
follows that (\ref{delta_n_exponent_limit_eq}) holds for this sequence
$\delta_n$.  This completes the proof of Theorem \ref{exponent_adaptation_thm_general_kernel}.

\appendix

\section{Details for Section \ref{achieving_bounds_sec_main}}\label{achieving_bounds_sec_append}

This appendix provides details for the results in Section
\ref{achieving_bounds_sec_main}.

\subsection{Critical Value}\label{cval_sec_append}

The critical value $c(j)=\cvalconst \sigma_n
2^{j/2}\sqrt{j}$ is justified by the following lemma.

\begin{lemma}\label{cval_lemma}
  Let $c(j)=\cvalconst \sigma_n 2^{j/2}\sqrt{j}$ and suppose that
  (\ref{kernel_lower_bound_assump}) and 
  (\ref{kernel_upper_bound_assump}) hold.  Then, if $\cvalconst$ is larger than
  a constant that depends only on the kernel $K$, we will have, for any sequence
  $\underline\ell_n\to\infty$,
  \begin{align*}
    P\left( |\hat f(t,j)-K_jf(t)|\le c(j)\text{ all }t\in[0,1], j\ge \underline\ell_n \right)\to 1.
  \end{align*}
\end{lemma}
\begin{proof}
Let $\mathbb{T}_n(t,j)=\sigma_n^{-1} 2^{-j/2} \left[ \hat f(t,j)-K_jf(t)
\right]=\int 2^{j/2}K(2^jt,2^jx)\, dW(x)$.  Note that the distribution of the process
$t\mapsto \mathbb{T}_n(2^{-j}(t+k))$ is the same for all $j,k,n$, since
  $cov\left( \mathbb{T}_n(2^{-j}(s+k),j), \mathbb{T}_n(2^{-j}(t+k),j) \right)
  =\int 2^jK(s+k,2^jx)K(t+k,2^jx)\, dx
  =\int K(s,u)K(t,u)\, du$,
using change of variables $u=2^jx-k$ and the fact that $K(t+k,u+k)=K(t,u)$.
Thus,
\begin{align*}
  &P\left( \sup_{t\in[0,1]} |\mathbb{T}_n(t,j)|>\cvalconst \sqrt{j} \right)
\le \sum_{k=0}^{2^j-1}P\left( \sup_{s\in[0,1]}
  |\mathbb{T}_n(2^{-j}(s+k),j)|>\cvalconst \sqrt{j} \right)  \\
&=2^j P\left( \sup_{t\in[0,1]}
  |\mathbb{T}_n(t,1)|>\cvalconst \sqrt{j} \right).
\end{align*}
By (\ref{kernel_upper_bound_assump}), we can apply
Theorem 8.1 in \cite{piterbarg_asymptotic_1996} to the process
$\mathbb{T}_n(t,1)$, which, along with
the tail bound $\Phi(-x)\le (x\sqrt{2\pi})^{-1}\exp\left( -x^2/2 \right)$
where $\Phi$ is the standard normal cdf,
gives the bound
$P\left( \sup_{t\in[0,1]} |\mathbb{T}_n(t,1)|>\cvalconst \sqrt{j} \right)\le
Cj^{1/\tau_K-1} \exp(-j\cvalconst/C)$
for some constant $C$ that depends only on the kernel $K$.
Thus,
\begin{align*}
&1-P\left( |\hat f(t,j)-K_jf(t)|\le c(j)\text{ all }t\in[0,1], j\ge \underline\ell_n \right)  \\
&\le \sum_{j=\underline\ell_n}^\infty 2^j P\left( \sup_{t\in[0,1]}
  |\mathbb{T}_n(t,1)|>\cvalconst \sqrt{j} \right)  \\
&\le \sum_{j=\underline\ell_n}^\infty 2^j C j^{1/\tau_K-1}\exp(-j\cvalconst/C)
= \sum_{j=\underline\ell_n}^\infty C j^{1/\tau_K-1}\exp(-j(\cvalconst/C-\log 2).
\end{align*}
For $\cvalconst>C\log 2$, this converges to $0$ as $n\to\infty$.

\end{proof}

\subsection{Confidence Interval for $\gamma$}\label{gamma_ci_sec_append}

We construct a confidence interval $[\hat\gamma_\ell,\hat\gamma_u]$ for $\gamma$,
which can be used in the confidence band described in Section
\ref{achieving_bounds_sec_main}.  The confidence interval covers $\gamma$ on the
event in (\ref{delta_coverage_eq}), so that the resulting cofidence band for $f$
contains $f$ on the event that (\ref{fxh_coverage_eq}) and
(\ref{delta_coverage_eq}) both hold.

Let
$\underline G(j_1,j_2)=\underline G(\varepsilon,\underline B,\overline
B,\underline\gamma,\overline\gamma,j_1,j_2)=\min_{B\in[\underline B,\overline
  B], \gamma\in [\underline\gamma,\overline\gamma]}B(\varepsilon -2^{-(j_2-j_1)\gamma})$
and
$\overline G(j_1,j_2)=\overline G(\underline B,\overline
B,\underline\gamma,\overline\gamma,j_1,j_2)=\max_{B\in[\underline B,\overline
  B], \gamma\in [\underline\gamma,\overline\gamma]}B(1 + 2^{-(j_2-j_1)\gamma})$.
Let
\begin{align*}
  \tilde\gamma_{\ell}(j_1,j_2)
  =\frac{\log_2 \underline G(j_1,j_2) - \log_2\left[ \hat\Delta(j_2,j_2)+\tilde c(j_1,j_2) \right]}{j_1}
\end{align*}
with the convention that $\tilde\gamma_{\ell}(j_1,j_2)=\underline\gamma$ when
$\underline G(j_1,j_2)\le 0$.
Let
\begin{align*}
  \tilde\gamma_u(j_1,j_2)
  =\frac{\log_2 \overline G(j_1,j_2) - \log_2\left[ \hat\Delta(j_2,j_2)-\tilde c(j_1,j_2) \right]}{j_1} 
\end{align*}
with the convention that $\tilde\gamma_u(j_1,j_2)=\overline\gamma$ when $\log_2\left[ \hat\Delta(j_2,j_2)-\tilde c(j_1,j_2) \right]\le 0$.
Let
\begin{align*}
  \hat\gamma_\ell=\max_{j\in\mathcal{J}_n}\tilde\gamma_\ell(j_1,j_2)
\text{ and }
\hat\gamma_u=\min_{j\in\mathcal{J}_n}\tilde\gamma_u(j_1,j_2).
\end{align*}
Then $\gamma\in[\hat\gamma_\ell,\hat\gamma_u]$ on the event in (\ref{delta_coverage_eq}). 
To see this, note that, 
by (\ref{delta_bound_eq}), we have, for all $j_1,j_2\in\mathcal{J}_n$
\begin{equation}\label{gamma_est_lower_bound_eq}
\begin{aligned}
  &2^{-j_1\gamma}\underline G(j_1,j_2)\le 2^{-j_1\gamma}B(\varepsilon -2^{-(j_2-j_1)\gamma})  
    \le \Delta(j_{1},j_{2};f)  \le \hat\Delta(j_1,j_2)+\tilde c(j_1,j_2),
\end{aligned}
\end{equation}
and
\begin{align*}
\hat\Delta(j_2,j_2)-\tilde c(j_1,j_2)
  \le \Delta(j_{1},j_{2};f)  
  \le 2^{-j_1\gamma}B(1 +2^{-(j_2-j_1)\gamma})  
  \le 2^{-j_1\gamma}\overline G(j_1,j_2).
\end{align*}
Taking logs and rearranging gives $\gamma\in [\tilde\gamma_\ell(j_1,j_2),\tilde\gamma_u(j_1,j_2)]$.
Note also that
\begin{align*}
  &\tilde\gamma_u(j_1,j_2)-\tilde\gamma_\ell(j_1,j_2)
  \le \frac{\log_2 \overline G(j_1,j_2)-\log_2 \underline G(j_1,j_2)}{j_1}
    + \frac{2\tilde c(j_1,j_2)}{j_1(\hat \Delta(j_1,j_2)-\tilde c(j_1,j_2))\log 2}  \\
  &\le \frac{\log_2 \overline G(j_1,j_2)-\log_2 \underline G(j_1,j_2)}{j_1}
    + \frac{2\tilde c(j_1,j_2)}{j_1(2^{-j_1\overline \gamma}\underline G(j_1,j_2)-2\tilde c(j_1,j_2))\log 2}
\end{align*}
where the first inequality uses $|\log a - \log b|\le |a-b|/\min\{a,b\}$ and the
second inequality uses (\ref{gamma_est_lower_bound_eq}).

Let $\tilde c(j_1,j_2)=\cvalconst \sigma_n 2^{j_1/2}\sqrt{j_1}+\cvalconst
\sigma_n 2^{j_2/2}\sqrt{j_2}$, so that Lemma \ref{cval_lemma} applies.
Let $j_1,j_2$ satisfy $j_1,j_2\to\infty$, $j_2-j_1\to\infty$, and $j_2/\log n\to
0$.  Then the above display is bounded by a constant times $j_1^{-1}$.  To see
this, note that 
$\underline G(j_1,j_2)$ and $\overline G(j_1,j_2)$ converge to positive
constants, and
$2^{j_1\overline\gamma}\tilde c(j_1,j_2)\to 0$ by the conditions on $j_1$ and $j_2$.

We collect these results in a theorem.

\begin{theorem}\label{gamma_ci_thm}
  Let $\hat\gamma_\ell$ and $\hat\gamma_u$ be given above.  Then, on the event
  in (\ref{delta_coverage_eq}), we have
  $\gamma\in[\hat\gamma_\ell,\hat\gamma_u]$ for
  $f\in\Fselfsim(\gamma,B,\varepsilon)$ with $B\in[\underline B,\overline B]$
  and $\gamma\in[\underline\gamma,\overline\gamma]$.
  Furthermore, if we take $\tilde c(j_1,j_2)=\cvalconst \sigma_n
  2^{j_1/2}\sqrt{j_1}+\cvalconst \sigma_n 2^{j_2/2}\sqrt{j_2}$ and
  $\mathcal{J}_n$ contains sequences
$j_1=j_{1,n}$ and $j_2=j_{2,n}$ which satisfy $j_1,j_2\to\infty$, $j_2-j_1\to\infty$, and $j_2/\log n\to
0$, then, for any sequence $r_n$ with $r_n\to 0$ and $r_n/j_1\to\infty$, we have
\begin{align*}
  \gamma-r_n\le \hat\gamma_\ell\le \gamma\le \hat\gamma_u
  \le \gamma+r_n
\end{align*}
with probability approaching one uniformly over
$\cup_{\gamma\in[\underline\gamma,\overline\gamma],B\in [\underline B,\overline
  B]}\FGN(\varepsilon,\varepsilon B,B)$.
\end{theorem}

\subsection{Length of the Confidence Band}\label{confidence_band_length_sec_append}

We now bound the length of this confidence band.
From (\ref{confidence_band_width_eq}), it follows that,
on the event
$\gamma-r_n\le \hat\gamma_\ell\le\gamma\le\hat\gamma_u\le \gamma+r_n$, the
length of the confidence band is bounded by
\begin{align*}
  \sup_{\gamma_u,\gamma_\ell\in [\gamma-r_n,\gamma+r_n]}
  \min_{j,j_1,j_2\in\mathcal{J}_n}  \left[
c(j) +\frac{B(2^{-j_1\gamma}+2^{-j_2\gamma})+2c(j_1)+2c(j_2)}{a(\varepsilon,j_1,j_2,j,\gamma_\ell,\gamma_u)} \right]
\end{align*}
where $c(j)=\cvalconst \sigma 2^{j/2}\sqrt{j/n}$.

It turns out that it will suffice to get an upper bound for the minimum in the
above display by taking
$j=j_{n,\gamma}=\lfloor \rho_\gamma + (2\gamma+1)^{-1} (\log_2 (n/\log_2 n))
\rfloor$,
$j_1=j_{1,n,\gamma}=j_{n,\gamma}-m_{1,n}$ and
$j_2=j_{2,n,\gamma}=j_{n,\gamma}-m_{2,n}$
where $m_{1,n}$ and $m_{2,n}$ are sequences such that
$m_{2,n}\to\infty$,
$m_{1,n}-m_{2,n}\to\infty$,
$r_n m_{1,n}\to 0$ and, for all $\gamma\in[\underline\gamma,\overline\gamma]$,
$j_{1,n,\gamma}\to \infty$
and $j_{2,n,\gamma}\to \infty$.
Applying the lemmas below gives the bound
\begin{align*}
  \left[ \frac{\overline c_K\sigma 2^{\rho_\gamma/2}}{(2\gamma+1)^{1/2}} 
  +B\varepsilon^{-1}2^{\gamma(1-\rho_\gamma)}  \right] (n/\log
n)^{-\gamma/(2\gamma+1)}[1+o(1)]
\end{align*}
where the $o(1)$ term is over
$\gamma\in[\underline\gamma,\overline\gamma]$, $B\in[\underline B,\overline B]$.
Setting
$\rho_\gamma = \log_2 \left(\sigma^{-1}
  B\varepsilon^{-1}\right)^{2/(2\gamma+1)}$
so that
$2^{\rho_\gamma/2}=\left(\sigma^{-1} B\varepsilon^{-1}\right)^{1/(2\gamma+1)}=\sigma^{2\gamma/(2\gamma+1)-1}\left(B\varepsilon^{-1}\right)^{1/(2\gamma+1)}$
gives
\begin{align*}
  \left[ \frac{\overline c_K}{(2\gamma+1)^{1/2}} 
    +2^{\gamma} \right]\sigma^{2\gamma/(2\gamma+1)}\left(B\varepsilon^{-1}\right)^{1/(2\gamma+1)}(n/\log
n)^{-\gamma/(2\gamma+1)}[1+o(1)].
\end{align*}
Since $\sigma_n^2\log (1/\sigma_n)=(\sigma^2/ n)\left( (1/2)\log n- \log \sigma
\right)=(1+o(1))(\sigma^2/2)(\log n)/n$, this gives
a bound of $(\sigma_n^2 \log (1/\sigma_n))^{\gamma/(2\gamma+1)}$ times a
constant that is bounded uniformly over $\gamma\le\overline\gamma$, as required.

\begin{lemma}
  \begin{align*}
    \sup_{\gamma\in[\underline\gamma,\overline\gamma]} \sup_{\gamma_\ell,\gamma_u\in[\gamma-r_n,\gamma+r_n]}
    \left| \frac{a(\varepsilon,j_{1,n,\gamma},j_{2,n,\gamma},j_{n,\gamma},\gamma_\ell,\gamma_u)}{a(\varepsilon,j_{1,n,\gamma},j_{2,n,\gamma},j_{n,\gamma},\gamma,\gamma)} - 1 \right|
    \to 0.
  \end{align*}
\end{lemma}
\begin{proof}
For $n$ large enough, we have, for any
$\gamma\in[\underline\gamma,\overline\gamma]$ and $\gamma_\ell,\gamma_u$ with
$\gamma-r_n\le \gamma_\ell\le\gamma_u\le \gamma+r_n$,
\begin{align*}
  &\varepsilon 2^{m_{1,n}(\gamma-r_n)} - 2^{m_{2,n}(\gamma+r_n)}
  \le a(\varepsilon,j_{1,n,\gamma},j_{2,n,\gamma},j_{n,\gamma},\gamma_\ell,\gamma_u)
  \le \varepsilon 2^{m_{1,n}(\gamma+r_n)} - 2^{m_{2,n}(\gamma-r_n)}
\end{align*}
and
$a(\varepsilon,j_{1,n,\gamma},j_{2,n,\gamma},j_{n,\gamma},\gamma,\gamma)=\varepsilon
2^{m_{1,n}\gamma} - 2^{m_{2,n}\gamma}$.  Thus,
\begin{align*} &\frac{a(\varepsilon,j_{1,n,\gamma},j_{2,n,\gamma},j_{n,\gamma},\gamma_\ell,\gamma_u)}{a(\varepsilon,j_{1,n,\gamma},j_{2,n,\gamma},j_{n,\gamma},\gamma,\gamma)}
  \le \frac{\varepsilon 2^{m_{1,n}(\gamma+r_n)} - 2^{m_{2,n}(\gamma-r_n)}}{\varepsilon
2^{m_{1,n}\gamma} - 2^{m_{2,n}\gamma}}  \\
  &= \frac{2^{m_{1,n}r_n} - \varepsilon^{-1} 2^{-m_{2,n}r_n+(m_{2,n}-m_{1,n})\gamma}}{1 - \varepsilon^{-1} 2^{(m_{2,n}-m_{1,n})\gamma}}
\end{align*}
which converges to one uniformly over
$\gamma\in[\underline\gamma,\overline\gamma]$ by the conditions on $m_{1,n}$ and
$m_{2,n}$.  The result follows from this and a similar argument with the lower bound.
\end{proof}

\begin{lemma}
  \begin{align*}
    \frac{2^{-\gamma j_{1,n,\gamma}} + 2^{-\gamma j_{2,n,\gamma}}}{a(\varepsilon,j_{1,n,\gamma},j_{2,n,\gamma},j_{n,\gamma},\gamma,\gamma)}
    = 2^{-\gamma j_{n,\gamma}}\varepsilon^{-1}(1+o(1))
  \end{align*}
  where the $o(1)$ term is uniform over all $\gamma\in[\underline\gamma,\overline\gamma]$.
\end{lemma}
\begin{proof}
We have 
\begin{align*}
  &\frac{2^{-\gamma j_{1,n,\gamma}} + 2^{-\gamma j_{2,n,\gamma}}}{2^{-\gamma j_{n,\gamma}}\varepsilon^{-1}a(\varepsilon,j_{1,n,\gamma},j_{2,n,\gamma},j_{n,\gamma},\gamma,\gamma)}
  =\frac{2^{-\gamma (j_{1,n,\gamma}-j_{n,\gamma})} + 2^{-\gamma (j_{2,n,\gamma}-j_{n,\gamma})}}{2^{m_{1,n}\gamma} - \varepsilon^{-1}2^{m_{2,n}\gamma}}  \\
  &=\frac{1 + 2^{-(m_{1,n}-m_{2,n})\gamma}}{1 - \varepsilon^{-1}2^{-(m_{1,n}-m_{2,n})\gamma}}
\end{align*}
which converges to one uniformly over
$\gamma\in[\underline\gamma,\overline\gamma]$ by the conditions on $m_{1,n}$ and $m_{2,n}$.
\end{proof}

\begin{lemma}
   If $\rho_{\gamma}$ is bounded over
  $\gamma\in[\underline\gamma,\overline\gamma]$, then 
  $c(j_{1,n,\gamma})/2^{-\gamma j_{1,n,\gamma}}\to 0$
  and $c(j_{2,n,\gamma})/2^{-\gamma j_{2,n,\gamma}}\to 0$ uniformly over $\gamma\in[\underline\gamma,\overline\gamma]$. 
  Furthermore, $c(j_{n,\gamma})\le \overline c_K\sigma
  2^{\rho_\gamma/2}(2\gamma+1)^{-1/2}(n/\log n)^{-\gamma/(2\gamma+1)}$ and
   $2^{-\gamma j_{n,\gamma}}\le 2^{\gamma(1-\rho_\gamma)}(n/\log_2 n)^{-\gamma/(2\gamma+1)}$.
\end{lemma}
\begin{proof}
We have
\begin{align*}
  &c(j_{n,\gamma})^2/(\overline c_K\sigma)^2= 2^{j_{n,\gamma}}j_{n,\gamma}/n  \\
   &= 2^{\lfloor \rho_\gamma + (2\gamma+1)^{-1} (\log_2 (n/\log_2 n))
 \rfloor} \lfloor (2\gamma+1)^{-1} (\log_2 n - \log_2 \log_2 n)
 \rfloor/n  \\
  &\le 2^{\rho_\gamma} 2^{(2\gamma+1)^{-1} (\log_2 (n/\log_2 n))} (2\gamma+1)^{-1} (\log_2 n)/n  
   =2^{\rho_\gamma}(2\gamma+1)^{-1}(n/\log_2 n)^{-2\gamma/(2\gamma+1)}.
\end{align*}
and
\begin{align*}
  &2^{-\gamma j_{n,\gamma}}
  = 2^{-\gamma \lfloor \rho_\gamma + (2\gamma+1)^{-1}\log _2 (n/\log_2 n) \rfloor}
  \le 2^{\gamma(1-\rho_\gamma)-\gamma (2\gamma+1)^{-1}\log _2 (n/\log_2 n)}  \\
  &= 2^{\gamma (1-\rho_\gamma)}(n/\log_2 n)^{\gamma/(2\gamma+1)}.
\end{align*} 
For any $m\ge \rho_\gamma$, we have
\begin{align*}
&c(j_{n,\gamma}-m)^2/(2^{-\gamma(j_{n,\gamma}-m)}\overline c_K\sigma)^2
  = 2^{(2\gamma+1)(j_{n,\gamma}-m)}(j_{n,\gamma}-m)/n  \\
  &\le 2^{\log_2(n/\log_2 n)-(m-\rho_{\gamma})(2\gamma+1)} (2\gamma+1)^{-1}(\log_2 n)/n
  = 2^{-(m-\rho_{\gamma})(2\gamma+1)} (2\gamma+1)^{-1}
\end{align*}
Setting $m=m_{1,n}\to\infty$ it follows that $c(j_{1,n,\gamma})/2^{-\gamma
  j_{1,n,\gamma}}\to 0$ uniformly over
$\gamma\in[\underline\gamma,\overline\gamma]$ and similarly for $j_{2,n,\gamma}$.

\end{proof}

\bibliography{../../../../../library}

\end{document}